\documentclass[a4paper,11pt, reqno]{amsart}
\usepackage[margin=3cm]{geometry}
\usepackage{hyperref}
\usepackage{amsmath,amssymb,amsfonts}%

\usepackage{enumitem}
\usepackage[all]{xy}
\usepackage{array}
\usepackage{tabularx}

\usepackage{multirow}

\def\R{\mathbb{R}}

\def\b{\mathfrak{b}}

\def\Z{\mathbb{Z}}

\def\A{\mathcal{A}}

\def\L{\mathcal{L}}

\def\K{\mathcal{E}}
\def\Q{\mathcal{K}}

\def\H{\mathcal{H}}

\def\N{\mathbb{N}}

\def\x{\bold{x}}
\def\y{\bold{y}}
\def\z{\bold{z}}
\def\a{\bold{a}}
\def\b{\bold{b}}

\def\e{\bold{e}}

\def\w{\bold{w}}

\def\t{\bold{t}}
\def\0{\bold{0}}
\def\1{\bold{1}}
\def\·{\bullet}
\def\<{\langle}
\def\>{\rangle}
\newcommand{\dsum}{\displaystyle\sum}
\newcommand{\dprod}{\displaystyle\prod}

\usepackage{adjustbox}
\usepackage{tikz}
\usepackage{pgfplots}
\pgfplotsset{compat=newest}
\usetikzlibrary{decorations.markings}

\usepackage{enumitem}
\usepackage{multicol}

\usepackage{float}

\def\Ext{{\rm Ext}}
\def\ba{{\boldsymbol{\alpha}}}
\usepackage{schemata}%
\newtheorem{theorem}{Theorem}
\newtheorem{remark}[theorem]{Remark}
\newtheorem{proposition}[theorem]{Proposition}
\newtheorem{definition}[theorem]{Definition}
\newtheorem{lemma}[theorem]{Lemma}
\newtheorem{cor}[theorem]{Corollary}
\newtheorem{example}[theorem]{Example}

\allowdisplaybreaks

\let\origmaketitle\maketitle
\def\maketitle{
  \begingroup
  \def\uppercasenonmath##1{} 
  \let\MakeUppercase\relax 
  \origmaketitle
  \endgroup
}

\begin{document}
\title[]{\huge On minimal extended representations of generalized power cones}

\author[V. Blanco \MakeLowercase{and} M. Mart\'inez-Ant\'on]{{\Large V\'ictor Blanco$^{\dagger,\ddagger}$ and Miguel Mart\'inez-Ant\'on$^{\dagger,\ddagger}$}\bigskip\\
{\large $^\dagger$Institute of Mathematics (IMAG)\\
$^\ddagger$Dpt. Quant. Methods for Economics \& Business\medskip\\
Universidad de Granada\bigskip\\
\texttt{vblanco@ugr.es}, \texttt{mmanton@ugr.es}}}

\date{\today}

\maketitle

\begin{abstract}
In this paper, we analyze minimal representations of $(p,\ba)$-power cones as simpler cones. We derive some new results on the complexity of the representations, and we provide a procedure to construct a minimal representation by means of second order cones in case $\ba$ and $p$ are rational. The construction is based on the identification of the cones with a graph, the mediated graph. Then, we develop a mixed integer linear optimization formulation to obtain the optimal mediated graph, and then, the minimal representation. We present the results of a series of computational experiments in order to analyze the computational performance of the approach, both to obtain the representation and its incorporation into a practical conic optimization model that arises in facility location.
\end{abstract}

\keywords{
second order cones, $p$-order cones, power cones, weighted geometric optimization, extended representations, conic programming, mediated graphs, combinatorial optimization.}

\subjclass[2010]{90C23, 90C25, 91B32, 90C10, 90C27, 90C35}

\section{Introduction}

The use of conic structures in mathematical optimization problems has been widely extended within the last few years. The development of interior point methods with specialized barriers has given room to efficiently solve problems involving cones with high practical interest. The incorporation of exact techniques to solve conic optimization problems in different off-the-shelf optimization solvers has allowed practitioners to use these tools to make decisions in different fields. This is the case of facility location~\cite{blanco2019ordered,blanco2022fairness,blanco2014revisiting,xue2000efficient}, supervised classification~\cite{blanco2020optimal,blanco2020lp}, network covering~\cite{blanco2024optimal}, power flow~\cite{chiang2007power}, radiotherapy~\cite{shirvani2015p}, portfolio selection~\cite{krokhmal_risk_2010}, geometric programming~\cite{avriel_extension_1971,nestorenko_optimization_2022,ota_geometric_2019}, among many others. Thus, a lot of theoretical results have emerged to understand the geometry of different representations of cones~\cite{hien_differential_2015,kian_minimal_2019,lu_decompositions_2020,roy_self-concordant_2022}. In terms of efficiency of optimization problems involving cones, different polyhedral approximations have been proposed~\cite{ben-tal_lectures_2001} in order to take advantage of linear programming solvers to compute approximate solutions to conic optimization problems. Nevertheless, in \cite{glineur2000computational} the author showed that interior points approaches have a better performance than the linear programming approximations proposed in \cite{ben-tal_lectures_2001}. In \cite{vielma2008lifted}, linear approximations are recommended for conic programs with integer variables.

In practice, among the vast variety of cones, some of them have been extensively analyzed since most of the cones can be \emph{rewritten} in terms of these basic cones ~\cite{lu_decompositions_2020}. These cones are the power cone and the exponential cone, whose efficient representation and incorporation also allow the incorporation of many other cones such as the second order cone, the $p$-order cone, geometric cones, the SDP cone, etc. As already stated by one of the most popular conic optimization solvers, Mosek~\cite{mosek}, most of the cones are actually represented based on these two types of cones. 

In this paper, we analyze a general type of power cone, the $(p,\ba)$-power cone, which is defined as:
$$
    \Q_p^{d_1 + d_2}(\ba) = 
    \{(\x,\z) \in \R^{d_1} \times \R^{d_2}_+ : \|\x\|_p\leq \z^{\ba}\}.
$$
for $p \in \R_+ \cup \{\infty\}$ ($p\geq 1$), $d_1, d_2 \in \Z_+$ ($d_2 \geq 1$), and $\ba \in  \{(\alpha_1, \ldots, \alpha_{d_2}) \in \R^+: \sum_{j=1}^{d_2} \alpha_j =1\}$. This cone has in particular cases the $p$-order cone and the power cone, and its efficient representation would allow us to solve optimization problems that naturally appear in different fields.

An extended representation of a cone (in general, of any set) is an equivalent reformulation of the set through its projection onto a higher dimensional space and its identification with a product of simpler sets. In general, different representations are valid for the same set. Thus, a natural measure of the goodness of a representation is the complexity of the representation, that is, the dimension of the space where the set is embedded and the number of simpler factor sets in the representation. In this paper, we derive different extended representations of the $(p,\ba)$-power cone by means of simpler power cones and study its complexity. In case $p$ and $\ba$ are rational, we provide extended representations of the cone as $3$-dimensional rotated quadratic cones. The main advantage of these simpler cones is that they are efficiently handled by most off-the-shelf optimization solvers such as Gurobi, CPLEX, or FICO. We deal with the problem of constructing optimal representations by solving a mathematical optimization problem. 

The analysis of representations for particular cases of these cones has been already addressed by different authors. The case $\ba=\left(\frac{s}{r},\frac{r-s}{r}\right)$ was analyzed in \cite{krokhmal_risk_2010,blanco2014revisiting} where the authors derived methods to represent the cone in a higher dimensional space with $O(\log_2(r))$ extra variables.  In \cite{morenko2013p}, an algorithm is proposed to represent the cones with the particular choice $\ba=\left(\frac{r_1}{2^m},\frac{r_2}{2^m}, \frac{r_3}{2^m}\right)$ with complexity $O(m)$. The general case of this cone, i.e., when  $\ba=\left(\frac{r_1}{2^m}, \ldots, \frac{r_d}{2^m}\right)$ where $m\in \N$ was studied in \cite{kian_minimal_2019} with a similar strategy. As far as we know, the general case, $\ba=\left(\frac{s_1}{\hat{s}}, \ldots, \frac{s_d}{\hat{s}}\right)$, has been only previously analyzed in \cite{wang2022weighted} where the authors propose a brute force enumeration approach as well as a heuristic algorithm for the computation of a minimal representation of the cone by second order cones.  The previous approaches analyzing cones and the obtained complexities are summarized in Table \ref{table:soa}.

\begin{table}[h]
\centering
\begin{tabular}{>{\centering\arraybackslash}m{2cm} >{\centering\arraybackslash}m{2.9cm} >{\centering\arraybackslash}m{3.5cm} >{\centering\arraybackslash}m{4cm}}
\hline
\textbf{Ref.} & \textbf{Cone} & \textbf{Methodology} & \textbf{Complexity}\\
\hline\hline
\cite{krokhmal_risk_2010} & $\Q_1^{1+2}\left(\frac{s}{r},\frac{r-s}{r}\right)$ & Tower of variables & $O(\log_2(r))$\\
\hline
\cite{blanco2014revisiting} & $\Q_1^{1+2}\left(\frac{s}{r},\frac{r-s}{r}\right)$ & Binary decomposition & $O(\log_2(r))$\\
\hline
\cite{morenko2013p} & $\Q_1^{1+3}\left(\frac{r_1}{2^m},\frac{r_2}{2^m}, \frac{r_3}{2^m}\right)$ & A successive algorithm & $O(m)$\\
\hline
\cite{kian_minimal_2019} & $\Q_1^{1+n}\left(\frac{r_1}{2^m},\ldots, \frac{r_d}{2^m}\right)$ & IP model and heuristic & $O(m)$\\
\hline
\cite{bental2001} & $\Q_1^{1+n}\left(\frac{s_1}{\hat{s}},\ldots, \frac{s_d}{\hat{s}}\right)$ & 
Tower of variables & $O(\hat{s})$\\
\hline
\cite{wang2022weighted} & $\Q_1^{1+n}\left(\frac{s_1}{\hat{s}},\ldots, \frac{s_d}{\hat{s}}\right)$ & A brute force and  heuristics & $O(\max\{\log_2(\hat{s}),d-1\})$\\
\hline
This paper & $\Q_1^{1+n}\left(\frac{s_1}{\hat{s}},\ldots, \frac{s_d}{\hat{s}}\right)$ & Math-optimization based approach& $O(\max\{\log_2(\hat{s}),d-1\})$\\
\hline
\end{tabular}
\caption{Summary of state-of-the-art. Complexity is identified with the number of cones  $\Q_1^{1+2}\left(\frac{1}{2},\frac{1}{2}\right)$ involved in the representation.}\label{table:soa}
\end{table}

Apart from the particular simplifications, this type of generalized power cone appears, for instance,  when modeling attraction between different bodies. According to the Universal Law of Gravitation by Newton, the attraction of two objects can be obtained as:
$$
F  = \frac{Gm_1m_2}{\|\x_1-\x_2\|^2_2}
$$
where $G$ is the gravity constant, $m_1$ and $m_2$ are the masses of the two bodies, and $\x_1$ and $\x_2$ their positions in the space. Thus, given a minimum threshold for the attraction between the two objects, $T$, the inequality modeling all the masses and positions with such a minimum attraction can be written as $\frac{Gm_1m_2}{\|\x_1-\x_2\|^2_2} \geq T$, or equivalently:
$$
\|\x\|^2_2 \leq \frac{G}{T} m_1 m_2, \x \in \R^{3}, m_1, m_2 \in \R_+,
$$
where $\x = \x_1 - \x_2$. Thus, the region defined by the inequality above is exactly one of the cones in our family of generalized power cones, specifically $\Q_2^{3+2}\left(\frac{1}{2},\frac{1}{2}\right)$. Clearly, the use of different distances between the two bodies in higher dimensional spaces ($\R^{d_1})$ and generalizing the product of two attributes (as the masses) to more than two, $d_2$, and aggregating them by using the weighted geometric mean of them, is a reasonable measure of attraction, whose representation results in the shape of the generalized power cones that we analyze here.

In this paper, we analyze this general cone, and apart from providing constructive approaches to simply represent the cones, we derive a novel approach to construct minimal representations of the cone as simpler second order cones. The main structures that we use for this construction are the so-called \textit{mediated graphs}. These graphs are related to the concept of mediated sets that were introduced in \cite{reznick1989forms}. The connections between these sets and the power cone were suggested in  \cite{wang2022weighted}. In this paper, these structures are also applied to analyze the Sum of Non-negative Circuits (SONC) problem~\cite{wang2022nonnegative} or the Sum of Squares (SOS) problem~\cite{iliman2016lower}. Here, we propose, for the first time, an exact mathematical optimization-based methodology to derive minimal representations of the power cone through the problem that we call the Minimum Cardinality Mediated Graph Problem (MCMGP). 

The rest of the paper is organized as follows. In section \ref{sec:2a} we set the notation and the definition of the main concepts along the paper. Section \ref{sec:3} is devoted to seeing possible extensions of the general $(p,\ba)$-power cone by means of simpler $\boldsymbol{\beta}$-power cones and analyzing the complexity between them. In section \ref{sec:4} we will show the optimal extended representation of $\ba$-power cone using second order cones, for this purpose we will introduce the idea of $\ba$-mediated graph, its relationship with these extended representations and address the problem of finding the $\ba$-mediated graph of minimal cardinality (MCMGP) via combinatorial programming. Section \ref{sec:5} is due to show the computational performance of our mathematical programming model of MCMGP on the one hand; and on the other hand, to introduce a covering problem based on gravity whose feasible region is a generalized power cone and how using an optimal extended representation affects to the computational results. Finally, in Section \ref{sec:6} we draw some conclusions and future research lines on the topic.

\section{Preliminaries}\label{sec:2a}

The main goal of this paper is to analyze the representations of $(p,\ba)$-power cones as a product of simpler cones in higher dimensional spaces. The formal definition of this representation is stated as follows. 

\begin{definition}[Extended representation]
    Given a set $Q \subset \R^d$, we say that a finite family $\K=\{Q_l\}_{l=1}^{\L_\K}$ with $Q_l \subset \R^{d_l}$, for $l=1, \ldots, \L_\K$ is an \textbf{extended representation} of $Q$ if there exist $m_\K \in \Z_+$  and projections $\pi_l: \R^d \times \R^{m_\K} \rightarrow \R^{d_{l}}$, for $l=1, \ldots, \L_\K$,  such that:
    $$
    \x \in Q \text{ if, and only if, }\exists \w \in \R^{m_\K} \mbox{such that } \pi_l(\x,\w) \in Q_l \text{ } \forall l=1, \ldots, \L_\K.
    $$
    
    We denote by ${\Ext}(Q)$ the set of families which are extended representations of $Q$.  
\end{definition}
    For any extended representation $\K$ of the set $Q$, the integer $m_\K$ in the above definition will be denoted as the  \textbf{extended dimension} of $\K$ to represent $Q$ and the integer  $\L_\K$ will be called the \textbf{cardinality of the representation}.

Note that this definition differs slightly from the notion of \emph{extended formulation} of a polytope that has been widely studied in combinatorial optimization~(see e.g. \cite{kaibel2015simple}). The key idea of our notion of extended representation is to lift the original set $Q$ onto an extended higher dimensional space, keeping the components of the original set $Q$, but allowing a simpler and more suitable representation at the price of adding some new variables and constraints. A clear illustration of this type of extension is the one commonly used to reformulate disjunctive linear programming problems. Let $Q={\rm conv}(\bigcup_{i=1}^n P_i)$ the convex hull of the union of $n$ polytopes $P_1, \ldots, P_n \subset \R^d$, each of them defined by means of linear inequalities, $P_i=\{\z \in \R^d: A_i \z\leq \b_i\}$ with $A_i \in \R^{m_i\times d}$, $\b_i\in \R^{m_i}$ for $i=1, \ldots, n$. While $Q$ is difficult to be represented as a set of inequalities in $\R^d$, it can be represented by the set:
\begin{align*}
Q^\prime = \{&(\z, \z_1, \ldots, \z_n, \lambda_1, \ldots, \lambda_n) \in \R^d \times \R^{d n} \times \R^n:\\
& \z - \sum_{i=1}^n \z_i =0,\\
& A_i \z_i - \b_i \lambda_i \leq 0,\forall i=1, \ldots, n, \\
&\sum_{i=1}^n \lambda_i =1, \\
& \lambda_i \geq 0, \forall i=1, \ldots, n\}.
\end{align*}
Defining the projection $\pi$ as the identity in $\R^d \times \R^{d n} \times \R^n$, it is known (see e.g., \cite{balas1998projection}) that $\z\in Q$ if and only if there exists $\z_1, \ldots, \z_n \in \R^d, \lambda_1, \ldots, \lambda_n \in \R$ such that $\pi(\z,\z_1, \ldots, \z_n,\lambda_1, \ldots, \lambda_n) \in Q^\prime$. Thus $\K=\{Q^\prime\}$ is an extended representation of $Q$ with $\L_\K=1$ and $m_\K=dn+n$.

An extended representation can be seen as a generalized intersection of the sets that are part of the extension, as stated in the following result whose proof is straightforward from the definition.
    \begin{proposition}
        Let $Q \subset \R^d$ and $\K=\{Q_l\}_{l=1}^{\L_\K}$ with $Q_l \subset \R^{d_l}$, for $l=1, \ldots, \L_\K$. Then, $\K \in \Ext(Q)$ if and only if there exist $m_\K \in \Z_+$  and projections $\pi_l: \R^d \times \R^{m_\K} \rightarrow \R^{d_{l}}$, for $l=1, \ldots, \L_\K$,  such that:
    $$
    Q={\rm proj}_d\left(\boldsymbol{\pi}^{-1}\left(\dprod_{l=1}^{\L_\K}Q_l \right)\right)={\rm proj}_d\left(\bigcap_{l=1}^{\L_\K}\left(\pi_l^{-1}(Q_l) \right)\right)\cong {\rm proj}_d\left(\bigcap_{l=1}^{\L_\K}\left(Q_l\times \R^{m_\K-d_l} \right)\!\!\right),
    $$
    where ${\rm proj}_d: \R^d\times \R^{m_\K}\to \R^d$ is the canonical projection of $\R^d$ and $\boldsymbol{\pi}=(\pi_1\times\cdots\times \pi_{\L_\K}).$ Here, $\cong$ stands for the bijection equivalence relation, i.e., $A \cong B$ if $A$ and $B$ are bijective sets.
    \end{proposition}
    
The constructions derived in this paper will be based on nesting sequential representations. The correctness of the global extended representation of this construction is based on the following result.
\begin{proposition}\label{prop:1}
Let $\K=\{Q_1, \ldots, Q_{\L_\K}\} \in {\Ext}(Q)$ and $\K_l = \{Q_{l1}, \ldots, Q_{l\L_{\K_l}}\} \in {\Ext}(Q_l)$, for all $l=1, \ldots, \L_\K$. Then, 
$$
\K' := \displaystyle\bigcup_{l=1}^{\L_\K} \K_l \in \Ext(Q).
$$
\end{proposition}
\begin{proof}
Let us define $m_{\K'}=m_\K + \sum_{l=1}^{\L_\K} m_{\K_l}$. For each $l=1, \ldots, \L_{\K}$ and $s = 1, \ldots, \L_{\K_l}$, and define $\pi_{ls}': \R^d \times \R^{m_{\K'}} \rightarrow \R^{d_{ls}}$ as:
$$
\pi_{ls}' = \pi_{ls} \circ  \left(\pi_l\times id_{m_{\K_l}}\right) \circ \left(id_{d \times m_\K}\times {\rm proj}_{m_{\K_l}}\right)
$$
where $\pi_l$ is the projection of the extended representation $\K$ for the set $Q_l$, and $\pi_{ls}$ the projection of the extended representation $\K_l$ for the set $Q_{ls}$. There we denote by ${\rm proj}_{m_{\K_l}}$ the projection from $\R^{\sum_{l=1}^{\L_\K} m_{\K_l}}$ onto the coordinates in $\R^{m_{\K_l}}$. The commutative diagram of the constructed projection is as follows:

\begin{center}
\centerline{\begin{xy}
     (0,20)*+{\R^d \times \R^{m_{\K'}}}="a"; 
    (70,20)*+{\R^d \times \R^{m_{\K}} \times \R^{m_{\K_l}}}="b";
     (70, 0)*+{\R^{d_l} \times \R^{m_{\K_l}}}="c";
     (0, 0)*+{\R^{d_{ls}}}="d";
     {\ar         "a";"b"}?*!/_8pt/{\tiny (id_{d \times m_\K}\times {\rm proj}_{m_{\K_\ell}})};
     {\ar@{>}    "c";"d"}?*!/^6pt/{\pi_{ls}};
     {\ar@{>} "b";"c"}?*!/^-30pt/{(\pi_l\times id_{m_{\K_l}})};
     {\ar@*{[|<5pt>]}@{>} "a";"d"}?*!/^8pt/{\pi_{ls}'};
\end{xy}}
\end{center}
Thus, with the above definitions, we get that:
$$
    \x \in Q \text{ if, and only if, }\exists \w \in \R^{m_{\K'}} \mbox{such that } \pi'_{ls}(\x,\w) \in Q_{ls} \text{ } \forall l=1, \ldots, \L_{\K'},
    $$
    is true. Considering that $\x\in Q$ if, and only if, there exist  $\w_0\in \R^{m_\K}$ verifying $\pi_l(\x,\w_0)\in Q_l$ for every $l=1,\ldots,\L_{\K}$ and $\w_l\in \R^{m_{\K_{l}}}$ with $l=1,\ldots,\L_{\K}$ satisfying $\pi_{ls}(\pi_{l}(\x,\w_0),\w_l)\in Q_{ls}$ for all $s=1,\ldots, \L_{\K_{l}}$; therefore taking $\w:=(\w_0,\w_1,\ldots,\w_{\L_{\K}})\in \R^{m_{\K'}}$ we get that $\x\in Q$ if, and only if, $\pi'_{ls}(\x,\w)\in Q_{ls}$ for every $l=1, \ldots, \L_\K\text{ and }s=1,\ldots, \L_{\K_{l}}$.
\end{proof}

In this paper, we analyze extended representations of a particular, but general, family of convex cones. In what follows we describe the cones that we study here.

\begin{definition}[$(p,\ba)$-Power cone]
    Given $p \in \R_+ \cup \{\infty\}$ ($p\geq 1$), $d_1, d_2 \in \Z_+$ ($d_2 \geq 1$), $\ba \in \Lambda_{d_2}$, the $(d_1+d_2)$-dimensional $(p,\ba)$-\textbf{power cone} is defined as:
    $$
    \Q_p^{d_1 + d_2}(\ba) = 
    \{(\x,\z) \in \R^{d_1} \times \R^{d_2}_+ : \|\x\|_p\leq \z^{\ba}\}.
    $$
    where $\|\cdot\|_p$ denotes the $\ell_p$-norm in $\R^{d_1}$, $\Lambda_d = \{(\alpha_1, \ldots, \alpha_d) \in \R^+: \sum_{j=1}^d \alpha_j =1\}$ is the $d$-dimensional simplex, and for any $\z\in \R^d_+$ and $\ba \in \Lambda_d$, $\z^{\ba} = \prod_{j=1}^{d} z_j^{\alpha_j} = z_1^{\alpha_1} \cdots z_d^{\alpha_d}$ is the weighted geometric mean of $\z$ with weights $\ba$.
\end{definition}

\begin{remark}
Some particular cases of the $(p,\ba)$-power cones are specially interesting. Specifically: 
\begin{itemize}
\item $\Q_p^{d_1+1} := \Q_p^{d_1+1}(1)$ is the $(d_1+1)$-dimensional $p$-\emph{order cone} for $p \geq 1$. In case $p=2$, $\Q_2^{d_1+1}$ is 
the \emph{second order cone}~\cite{lobo1998second}.
\item $\Q^{1+d_2}(\ba):=\Q_1^{1+d_2}(\ba)$ is the $\ba$-\emph{power cone}~\cite{chares_cones_nodate}.
\item $\Q_p^{d_1+2}(\alpha) := \Q_p^{d_1+2}(\alpha,1-\alpha)$ is the $(d_1+2)$-dimensional $(p,\alpha)$-\emph{power cone}, for any $p\geq 1$ and $0\leq \alpha\leq 1$. In case $d_1=1$, $\Q^{1+2}(\alpha):=\Q_1^{1+2}(\alpha)$ is the $\alpha$-\emph{power cone}~\cite{koecher1957positivitatsbereiche}.
\end{itemize}
\end{remark}

\section{Extended representations of the generalized power cone}\label{sec:3}

In this section, we provide valid extended representations of the $(p,\ba)$-power cone by means of simpler power cones. We analyze the complexities of the different representations.

The first result is based on splitting the constraint defining the $(p,\ba)$-cone in two simple constraints (with a single linking variable), being the obtained constraints of simpler cones, in this case, a $p$-order cone and an $\ba$-power cone. As a consequence, being able to represent the original power cone stems from adequately representing (separately) each of these two cones.
\begin{lemma}\label{lem:4}
Let $p \in \R_+ \cup \{\infty\}$ ($p\geq 1$), $d_1, d_2 \in \R_+$ ($d_2 \geq 1$), $\ba \in \Lambda_{d_2}$. Then, $\K = \{\Q_p^{d_1+1}, \Q^{1+d_2}(\ba)\} \in \Ext(\Q_p^{d_1 + d_2}(\ba))$ with $m_\K=1$ and $\L_\K=2$.
\end{lemma}

\begin{proof}
The proof is straightforward.
\end{proof}

In what follows we provide extended representations of $p$-order cones and $\ba$-power cones.

The first result, which has been already proven in \cite{bental2001,krokhmal_risk_2010} for $p=2$, is called the \emph{tower of variables} representation, and allows one to reformulate a $p$-order cone in dimension $d+1$ as $d-1$ $p$-order cones in dimension $3$ using $d-2$ new variables.

\begin{lemma}\label{l:tov}
  Let $p \in \R_+ \cup \{\infty\}$ ($p\geq 1$), and $d \in \R_+$. Then, $\K=\{\Q_p^{2+1}, \stackrel{d-1)}{\ldots}, \Q_p^{2+1}\} \in {\Ext}(\Q_{p}^{d+1})$ with $m_\K=d-2$ and $\L_\K=d-1$.
\end{lemma}
\begin{proof}
The construction of the $3$-dimensional cones to represent the $p$-order cone can be done using a tree-shaped structure. Recall that the cone $\Q_p^{d+1}$ is defined by an inequality $\|(x_1, \ldots, x_d)\|_p \leq w$, or equivalently:
$$
\dsum_{j=1}^{d} x_j^p \leq w^p.
$$
Let $T$ be binary a tree with vertices indexed by the subsets of $\{1, \ldots, d\}$, $2^{\{1, \ldots, d\}}$, and being the root node $R_0=\{1, \ldots, d\}$ and its leaves $\{1\}, \ldots, \{d\}$. At each level of the tree, the parent index set is partitioned into two disjoint sets until the singleton nodes are reached. 

For each non-leave node, $v \in 2^{\{1, \ldots, d\}}$, in the tree we denote by $c_1(v)$ and $c_2(v)$ its unique children nodes. Then, the following constraints equivalently represent the constraint defining the cone:

\begin{align*}
    w_v \geq \|(w_{c_1(v)}, w_{c_2(v)})\|_p, \forall v \in T \text{ (non leave nodes)}
\end{align*}

where node $w_{\{i\}}$ is identified with the original variable $x_i$, and the root node variable $w_{\{1, \ldots, d\}}$ is identified with the variable $w$. Thus, there are as many constraints as non-leave nodes in the tree ($d-1$) and as many variables as non-root and non-leave nodes in the tree ($d-2$).

\end{proof}

The proof of the above result is constructive and allows an easy representation of the $p$-order cone by means of a binary tree.  In Figure \ref{fig:tov} we show a tree for $d=7$, which results in the following representation of the $p$-order cone $\Q^{7+1}_p$.
\begin{align*}
    w &\geq \|(w_{\{1,2,3\}}, w_{\{4,5,6,7\}})\|_p, \\
    w_{\{1,2,3\}} &\geq \|(w_{\{1,2\}}, x_3)\|_p, \\
        w_{\{1,2\}} &\geq \|(x_1, x_2)\|_p, \\
    w_{\{4,5,6,7\}} &\geq \|(w_{\{4,5,6\}}, x_7)\|_p, \\
    w_{\{4,5,6\}} &\geq \|(w_{\{4,5\}},x_6)\|_p,\\
        w_{\{4,5\}} &\geq \|(x_4,x_5)\|_p
\end{align*}

\begin{figure}[h]
\begin{center}
\begin{tikzpicture}[
    level 1/.style = {sibling distance = 4.5cm},
    level 2/.style = {sibling distance = 2.5cm},
    level 3/.style = {sibling distance = 1.5cm},
    level 4/.style = {sibling distance = 1cm},
    level distance=30pt
]
\node {$\{1, 2, 3, 4, 5, 6, 7\}$}
    child {node {$\{1,2,3\}$}
        child {node {$\{1,2\}$}
            child {node {$\{1\}$}}
            child {node {$\{2\}$}}
              }
        child {node {$\{3\}$}}
          }
    child {node {$\{4,5, 6, 7\}$}
        child { node {$\{4,5, 6\}$}
            child {node {$\{4,5\}$}
                child {node {$\{4\}$}}
                child {node {$\{5\}$}}
                }
                child {node {$\{6\}$}}
                    }
            child {node {$\{7\}$}}    
            };
\end{tikzpicture}
\end{center}
\caption{Tree for the tower of variables construction of Lemma \ref{l:tov} for $d=7$.\label{fig:tov}}
\end{figure}
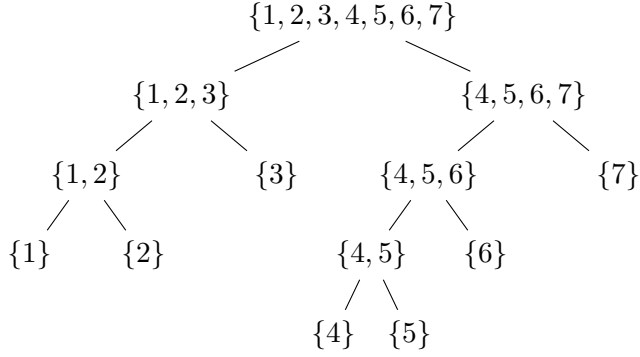

In the next result, we show that each of the $3$-dimensional cones that are part of the extended representation of the $p$-order cone can be represented as a family of $\frac{1}{p}$-power cones and a half-space.

We denote by $\H^{d+1} = \left\{(\x,w) \in \R^{d}\times \R: \dsum_{i=1}^{d} x_i \leq w\right\}$.
\begin{lemma}\label{lem:6}
    Let $p \in \R_+ \cup \{\infty\}$ ($p\geq 1$), and $d \in \R_+$. Then, $\K=\Big\{\Q^{1+2}\left(\frac{1}{p}\right), \stackrel{d)}{\ldots}, \Q^{1+2}\left(\frac{1}{p}\right), \H^{d+1}\Big\} \in \Ext(\Q_p^{d+1})$ with $m_\K=d$ and $\L_\K=d+1$.
\end{lemma}

\begin{proof}
The proof is based on that, for $w>0$, the inequality defining the cone $\Q_p^{d+1}$, $\|\x\|_p \leq w$ is equivalent to:
$$
\dsum_{j=1}^{d} \frac{x_j^p}{w^{p-1}} \leq w
$$
Let $q$ be the conjugate of $p$, i.e., $\frac{1}{p}+\frac{1}{q}=1$. Since $p-1=\frac{p}{q}$, we get that be above equation is equivalent to:
$$
\dsum_{j=1}^{d} \frac{x_j^p}{w^{\frac{p}{q}}} \leq w
$$
Which can be split into the following inequalities by adding auxiliary variables \linebreak $t_1,\ldots, t_{d} \geq 0$:
\begin{align}
\dsum_{j=1}^{d} t_ j &\leq w, \\
x_j^p &\leq t_j w^{\frac{p}{q}}, \forall j=1, \ldots, d. \label{eq:p}
\end{align}
Constraints \eqref{eq:p} can be rewritten as:
$$
x_j \leq t_j^{\frac{1}{p}} w^{\frac{1}{q}}.
$$
Thus, $\Big(x_j, (t_j, w)\Big) \in \Q^{1+2}(\frac{1}{p})$, for $j=1, \ldots, d$.

In case $w=0$, the proof is straightforward.
\end{proof}

Finally, by Proposition \ref{prop:1}, and the previous lemmas we get that one can construct an extended representation of the $(d_1+d_2)$-dimensional $(p,\ba)$-power cone by means of power cones and half-spaces.
\begin{cor}\label{cor:8}
    Let $p \in \R_+ \cup \{\infty\}$ ($p\geq 1$), $d_1, d_2 \in \R_+$ ($d_2 \geq 1$), $\ba \in \Lambda_{d_2}$. Then,
    $$
    \K \!= \!\left\{\Q^{1+2}\left(\frac{1}{p}\right), \stackrel{2(d_1-1)}{\ldots}, \Q^{1+2}\left(\frac{1}{p}\right), \H^{2+1}, \stackrel{d_1-1)}{\ldots}, \H^{2+1}, \Q^{1+d_2}(\ba)\right\}\!\in\!\Ext(\Q_p^{d_1 + d_2}(\ba))
    $$
    with $m_\K=3d_1-3$ and $\L_\K=3d_1-2$.
\end{cor}

Here, the original cone $\Q_p^{d_1+d_2}(\ba)$ is first represented by a $p$-order cone and a $\ba$-power cone. Then, the $p$-order cone is represented by $d_1-1$ three-dimensional $p$-order cones. Finally, each of these cones is represented as simpler three-dimensional $\frac{1}{p}$-power cones.

The above representation of the $(p,\ba)$-power cone requires embedding the cone in a higher dimensional space with $3d_1-3$ extra dimensions and using $3d_1-2$ sets in the representation (power cones and half-spaces). In the following results, we provide alternative representations reducing the complexity of the representations.

\begin{theorem}\label{th:9}
Let $p \in \R_+ \cup \{\infty\}$ ($p\geq 1$), $d_1, d_2 \in \R_+$ ($d_2 \geq 1$), $\ba \in \Lambda_{d_2}$. Then, 
    $$
    \K \!= \!\!\left\{\Q^{1+d_2}(\ba),\H^{d_1+1}, \Q^{1+d_2+1}\left(\frac{1}{p},\frac{1}{q}\ba\right),\stackrel{d_1)}{\ldots},\Q^{1+d_2+1}\left(\frac{1}{p},\frac{1}{q}\ba\right)\right\} \in\!\Ext(\Q_p^{d_1 + d_2}(\ba))
    $$
 with $m_\K=d_1+1$ and $\L_\K=d_1+2$.
\end{theorem}

\begin{proof}
    For the first extension, it is enough to apply Lemma \ref{lem:6} replacing $w$ by $\z^{\ba}$ in the proof. Hence, we have that
    $$
    \left\{\Q^{d_1+d_2}_1(\ba),\Q^{1+d_2+1}\left(\frac{1}{p},\frac{1}{q}\ba\right),\stackrel{d_1)}{\ldots},\Q^{1+d_2+1}\left(\frac{1}{p},\frac{1}{q}\ba\right)\!\right\}
    \in \Ext(\Q_p^{d_1 + d_2}(\ba)).
    $$
    To represent the cones $\Q^{d_1+d_2}_1(\ba)$ in the extension, i.e., elements $(\t,\z) \in \R^{d_1} \times \R^{d_2}$ such that $\|\t\|_1\leq \z^{\ba}$, we split the constraint as:
    \begin{align}
    \|\t\|_1\leq w, \label{eq:2.1}\\
    w \leq \z^{\ba}.\label{eq:2.2}
     \end{align}
    where \eqref{eq:2.1} define the half-space $\H^{d_1+1}$ and \eqref{eq:2.2} the cone  $\Q^{1+d_2}(\ba)$. Thus, $\{\Q^{1+d_2}(\ba)$, $\H^{d_1+1}\}\in \Ext(\Q^{d_1+d_2}_1(\ba))$ with extended dimension $1$ (a single extra variable to link the constraints). 
    
    Applying Proposition \ref{prop:1}, we get the desired result.
\end{proof}

The above result provides an alternative extension requiring embedding the cone into a $d_1+1$ extra-dimensional space by means of $d_1+2$ sets (cones and half-spaces).

Finally, we provide the simplest extended representation of the cone. 

\begin{theorem}\label{th:10}
Let $p \in \R_+ \cup \{\infty\}$ ($p\geq 1$), $d_1, d_2 \in \R_+$ ($d_2 \geq 1$), $\ba \in \Lambda_{d_2}$. Then,
    $$
    \K = \left\{\Q^{1+2}\left(\frac{1}{p}\right), \stackrel{d_1)}{\ldots}, \Q^{1+2}\left(\frac{1}{p}\right), \H^{d_1+1}, \Q^{1+d_2}(\ba)\right\} \in \Ext(\Q_p^{d_1 + d_2}(\ba))
    $$
    with $m_\K=d_1+1$ and $\L_\K=d_1+2$.
\end{theorem}

\begin{proof}
    The proof is straightforward from lemmas \ref{lem:4} and \ref{lem:6}, and Proposition \ref{prop:1}.
\end{proof}

\begin{figure}[h]
\begin{minipage}[t]{\textwidth}
\hspace{1.5cm} 
\schema{
   \schemabox{$\Q^{d_1+d_2}_p(\ba) \ \Big[\begin{array}{l} m_\K=1 \\ \L_\K=2\end{array}\Big]$}
}{
    \schemabox{
        \schema{\schemabox{$\Q^{d_1+1}_p \ \Big[\begin{array}{l} m_{\K_1}=d_1 \\ \L_{\K_1}=d_1+1\end{array}\Big]$}}{
        \schemabox{$\Q^{1+2}\Big(\frac{1}{p}\Big)$ \\ $\cdot$ \\ $\cdot \ \scriptstyle{d_1)}$ \\ $\cdot$ \\ $\Q^{1+2}\Big(\frac{1}{p}\Big)$ \\ $\H^{d_1+1}$}
        }
    \\ $\Q^{1+d_2}(\ba)$}
}
\end{minipage}
\caption{Extended representation of Theorem \ref{th:10}.\label{fig:C}}
\end{figure}

In Figure \ref{fig:C} we illustrate the extended representation of the cone derived from the above result. A summary of the results derived above is shown in Table \ref{table:complexity1}. Note that unless $d_1=1$, the representation in Theorem \ref{th:10} is strictly stronger than the one in Corollary \ref{cor:8}.

{\small
\begin{table}[h]
	\adjustbox{width=\textwidth}{\begin{tabular}{>{\centering\arraybackslash}p{1.9cm}>{\centering\arraybackslash}p{1cm}>{\centering\arraybackslash}p{0.99cm}>{\centering\arraybackslash}p{1.5cm}>{\centering\arraybackslash}p{2.7cm}>{\centering\arraybackslash}p{1.3cm}>{\centering\arraybackslash}p{0.95cm}>{\centering\arraybackslash}p{.85cm}}
			Result & $m_\K$ & $\L_\K$ & $\Q^{1+d_2}(\ba)$ & $\Q^{1+d_2+1}\left(\frac{1}{p},\frac{1}{q}\ba\right)$ &  $\Q^{1+2}(\frac{1}{p})$ & $\H^{d_1+1}$ & $\H^{2+1}$\\\hline\hline
			Corollary \ref{cor:8} & $3d_1-3$ & $3d_1-2$ & $1$& $-$ & $2(d_1-1)$ & $-$ & $d_1-1$\\\hline
			Theorem \ref{th:9} & $d_1+1$ & $d_1+2$ & $1$ & $d_1$ &$-$ & $1$ & $-$\\ \hline
			Theorem \ref{th:10} & $d_1+1$ & $d_1+2$ & $1$ & $-$ &$d_1$ & $1$ & $-$\\ \hline
	\end{tabular}}
	\caption{Summary of main results from Section \ref{sec:3}. \label{table:complexity1}}
\end{table}
}

\begin{remark}\label{rem:12}
Note that one can analogously represent more general cones in the form:
$$
\Q_p(\ba)[f,g,h] := \Big\{(\x,\z,\t) \in \R^{d_1}\times \R^{d_2} \times \R^{d_3}: \|f(\x)\|_p\leq h(\z)^{\ba}  + g(\t)  \Big\}
$$
where $f: \R^{d_1}\rightarrow \R^{d_1'}$, $g: \R^{d_3} \rightarrow \R$, and $h: \R^{d_2} \rightarrow \R^{d_2'}$ are linear functions, by adding to the representations above $d_1+d_2+1$ extra variables and linear equations (hyperplanes).
\end{remark}

\section{Optimal SOC representations of the rational power cone}\label{sec:4}

The extended representations derived in the previous section allow one to reduce the representation of the complex cone $\Q_p^{d_1+d_2}(\ba)$ by means of simpler power cones $\Q^{1+d_2} (\ba)$ for $\ba\in \Lambda_{d_2}$. For the sake of simplification, in the rest of the paper, we denote $d=d_2$. The main implications of these representations come from the fact that one can restrict the variables of an optimization problem to belong to a $(p,\ba)$-power cone, but formulate the constraints as constraints in the form $x \leq \z^{\ba}$ for some $\ba \in \Lambda_d$. In case the components of $\ba$ are rational ($\ba \in \mathbb{Q}^{d}$, or equivalently, $\ba = (\frac{s_1}{\hat s}, \ldots, \frac{s_d}{\hat s})\in \Lambda_d$ with $s_1, \ldots, s_d \in \Z_+$, $\gcd(s_1, \ldots, s_d)=1$, and $\hat{s} = \sum_{j=1}^d s_j$), one can go further and derive representations by means of second order cones.  In this section, we will present the main contribution of this paper, a mathematical optimization-based approach to construct a minimal extended representation of rational power cones (in terms of the cardinality of the representation and extended dimension).

The mathematical optimization model that we propose to construct a minimal extended representation of the cone is based on the notion of $\ba$-mediated graph that we define as follows.

\begin{definition}[$\ba$-Mediated graph]\label{def:mediatedset}
Let $\ba = \frac{1}{\hat{s}} (s_1, \ldots, s_d) \in \Lambda_d$ and $\A_{\ba} = \{\hat{s} \e_j: j = 1, \ldots, d-1\} \cup \{\0\}\subset \R^{d-1}$ (here $\e_j$ stands for the $j$th unit vector in $\R^{d-1}$). An $\ba$-\textbf{mediated graph} is a digraph $G = (\A_{\ba} \cup X, A)$ with $X \subseteq \R^{d-1}$ such that:
    \begin{enumerate}
        \item $\b_{\ba} = (s_1,\ldots, s_{d-1}) \in X$, and for each $\x\in X$ there exist $\y\neq \z\in \A_{\ba}\cup X$ such that $\x = \frac{1}{2}(\y+\z)$.
        \item $A = \Big\{(\x, \y): \exists \z \in (\A_{\ba}\cup X)\backslash\{\y\} \text{ such that } \x = \frac{1}{2}(\y+\z)\Big\}$.
    \end{enumerate}
\end{definition}

We denote by $\mathcal{G}_\ba$ the set of $\ba$-mediated graphs. For each $G = (\A_{\ba} \cup X,A) \in \mathcal{G}_\ba$ we call the set $X$, an $\ba$-\textbf{mediated set}. $G$ is said a \emph{minimal} $\ba$-mediated graph if the cardinality of its mediated set is minimal.

In Figure \ref{fig:mediated} we show a small example of $\ba$-mediated graph for $\ba = (\frac{1}{6},\frac{2}{6},\frac{3}{6})$. In this case, $\A_{\ba} = \{(6,0), (0,6), (0,0)\}$ (red dots) and $\b_\ba = (1,2)$ (green dot). The set of mediated nodes in the plot is $X= \{(1,2), (2,4), (4,2)\}$ (blue dots). Note that $\b_{\ba} = (1,2) \in X$ and for each of the elements $\x$ in $X$ there exist two different points in $\A_\ba \cup X$ (the arrows in the plot from each of the elements in $X$) such that $\x$ is the average of those vectors. For instance $(1,2) = \frac{1}{2} ((0,0) + (2,4))$, $(2,4) = \frac{1}{2} ((0,6) + (0,0))$, and $(3,2) = \frac{1}{2} ((0,4) + (6,0))$.

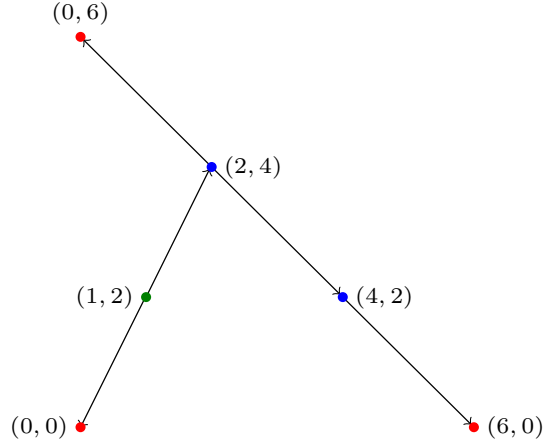
\begin{figure}
\centering
\adjustbox{width=0.5\textwidth}{
\begin{tikzpicture}

\definecolor{darkgray176}{RGB}{176,176,176}
\definecolor{green}{RGB}{0,128,0}

\pgfplotsset{every axis plot/.append style={very thin}}

\begin{axis}[
hide x axis,
hide y axis,
tick align=outside,
tick pos=left,
title={$\mathbf{\ba}=(1,2,3)/6$},
xmin=-1, xmax=7,
xtick style={color=black},
ymin=-1, ymax=7,
axis equal,
]

\node[left] at (0,0) {\tiny $(0,0)$};
\node[left] at (1,2) {\tiny $(1,2)$};
\node[right] at (4,2) {\tiny $(4,2)$};
\node[right] at (2,4) {\tiny $(2,4)$};
\node[above] at (0,6) {\tiny $(0,6)$};
\node[right] at (6,0) {\tiny $(6,0)$};

\draw[<-] (0,0)--(0.97,1.96);
\draw[<-] (1.97,3.97)--(1,2);
\draw[->] (2.04,3.96)--(3.96,2.04);
\draw[->] (4.04,1.96) -- (5.96,0.04);
\draw[->] (1.96,4.04) -- (0.04, 5.96);
\addplot [semithick, red, mark=*, mark size=1.25, mark options={solid}, only marks]
table {%
6 0
};
\addplot [semithick, red, mark=*, mark size=1.25, mark options={solid}, only marks]
table {%
0 6
};
\addplot [semithick, red, mark=*, mark size=1.25, mark options={solid}, only marks]
table {%
0 0
};
\addplot [semithick, green, mark=*, mark size=1.25, mark options={solid}, only marks]
table {%
1 2
};
\addplot [semithick, blue, mark=*, mark size=1.25, mark options={solid}, only marks]
table {%
2 4
};
\addplot [semithick, blue, mark=*, mark size=1.25, mark options={solid}, only marks]
table {%
4.00000000000003 2.00000000000005
};
\end{axis}

\end{tikzpicture}}
\caption{A $\left(\frac{1}{6},\frac{2}{6},\frac{3}{6}\right)$-mediated graph. \label{fig:mediated}}
\end{figure}

Given an $\ba$-mediated graph, $G=(\A_{\ba} \cup X,A)$, note that all the nodes belong to the convex hull of $\A_{\ba}$. Furthermore, by definition of $\ba$-mediated graph, if $(\x, \y), (\x, \z)$ are arcs in $G$, then $\boldsymbol{\mu}^\x = \frac{1}{2} (\boldsymbol{\mu}^\y + \boldsymbol{\mu}^\z)$, where $\boldsymbol{\mu}^\t \in \Lambda_d$ are the coordinates of $\t \in \A_{\ba} \cup X$ with respect to the elements in $\A_\ba$, i.e., $\t = \sum_{j=1}^{d-1} \mu_j^\t \hat{s} \e_j + \mu_d^\t \0$.

The following result states the correspondence between the extended representations and the above defined mediated graphs.

\begin{theorem}
    There exists a \emph{one-to-one} correspondence between $\mathcal{G}_\ba$ and the set of extended representations by second order cones of $\Q^{1+d}(\ba)$. 
\end{theorem}

\begin{proof}
 Let us first proof any $\ba$-mediated graph $G=(\A_{\ba} \cup X, A)$ with $X=\{\x_1,\ldots,\x_\L\}$ induces an extended representation of $\Q^{1+d}(\ba)=\{x\leq \z^{\ba}\}$ using second order cones  $\Q^{1+2}(\frac{1}{2})$. 
    
    Let us define the set of indices $\mathbb{L}=\{1,\ldots,\L+d\}$ such that $\x_\L=\b_{\ba}$, $\x_{\L+i}=\hat{s}\e_i$ for $i=1,...,d-1$, and $\x_{\L+d}=\0$. Additionally, we denote by $\sigma(j) \subset \mathbb{L}$ (with $|\sigma(j)|=2$) the set of indices for the $j$-th mediated element such that $\x_j=\frac{1}{2}\sum_{\ell\in\sigma(j)}\x_\ell$, for $j=1,\ldots, \L$. 
    
    With the above notation, we have that:
    $$
    \pi_j(\bar{\w})=(\bar{w}_j,\bar{\w}_{\sigma(j)}),  \forall j=1,\ldots,\L
    $$
    where $\bar{\w}=(w_1,\ldots,w_{\L-1},x,z_1,\ldots,z_d)$, define an extended second order cone representation of the cone $\Q^{1+d}(\ba)$. 
    
   Let us consider consider $(x,\z)\in \Q^{1+d}(\ba)$, that is, $x\leq \z^{\ba}$. As already mentioned, there always exists $\w\in \R^{\L-1}$ verifying $w_j\leq \z^{\boldsymbol{\mu}^j}$ for $j=1,\ldots, \L-1$, where $\boldsymbol{\mu}^j$ are the coordinates of $\x_j$ as point in the convex hull of $\A_{\ba}$ and
    $$
    \bar{w}_j\leq \dprod_{\ell\in\sigma(j)}\bar{w}_{\ell}^{\frac{1}{2}}.
    $$

    On the other hand, if $\pi_j(\bar{\w})=(\bar{w}_j,\bar{\w}_{\sigma(j)})\in \Q^{1+2}\left(\frac{1}{2}\right), \forall j=1,\ldots,\L.$ In particular, $\pi_\L(\bar{\w})=(x,\bar{\w}_{\sigma(\L)})\in \Q^{1+2}\left(\frac{1}{2}\right)$ and then: $$\dprod_{\ell\in\sigma(\L)}\bar{w}_\ell^{\frac{1}{2}}\leq \dprod_{\ell\in\sigma(\L)}\z^{\frac{1}{2}\boldsymbol{\mu^\ell}}=\z^{\frac{1}{2}\dsum_{\ell\in\sigma(\L)}\boldsymbol{\mu}^\ell}=\z^{\ba}.$$
    Thus, $(x,\z)\in \Q^{1+d}(\ba).$ 
    
    Note that the above proof is fully reversible, so any second order cone extended representation for $\Q^{1+d}(\ba)$ can be identified with an element in $\mathcal{G}_\ba$.
\end{proof}

The above result does not just relate mediated graphs with SOC representations of $\Q^{1+d}(\ba)$, this makes the problem of finding an $\ba$-mediated graph of the minimum cardinality equivalent to the problem of finding the minimum SOC representation. Moreover, the cardinality of the representation is the cardinality of the mediated set in the mediated graph.

\begin{remark}[Lower and Upper bounds for the cardinality of a $\ba$-mediated graph]\label{bounds}
In \cite{wang2022weighted}, the author proves two lower bounds for the number of nodes in a minimal $\ba$-mediated graph, one based on the dimension of the nodes ($d$) and other in the sum of the elements $s_1, \ldots, s_d$, $\hat{s}$, being then a suitable lower bound for the $\ba$-mediated set the maximum of both values, $\Delta_{\ba} = \max \{d-1, \lceil\log_2 (\hat{s})\rceil\}$. 

On the other hand, as already mentioned in \cite{blanco2014revisiting}, one can overestimate the size of the representation for the cone by using a binary representation of the elements $s_1, \ldots, s_d$. Specifically, for $a \in \mathbb{Z}_+$, we denote by $\Omega(a)$ the set of exponents with strictly positive coefficients in the binary decomposition of $a$, i.e., if $a = \sum_{k=0}^K b_k 2^k$, with $b_k \in \{0,1\}$ for $k=1, \ldots, K$, then $\Omega(a)=\{k : b_k=1\}_{k=0}^K$. Then, an upper bound for the minimum size of an $\ba$-mediated set is $\dsum_{i=1}^d|\Omega(s_i)|+|\Omega(2^{\lceil\log_2 (\hat{s})\rceil}-\hat{s})|-1$. The representations reaching this upper bound are the most common when representing $p$-order cones in several applications in the literature (see e.g. \cite{blanco2022fairness,blanco2024optimal,blanco2014revisiting}).
\end{remark}

\subsection{A mathematical optimization model to construct minimal mediated graphs}\label{model}

Now, we are in conditions to present the mathematical optimization model that we propose to construct minimal $\ba$-mediated graphs, and then also to obtain minimal representations of power cones. We call the problem the {\bf Minimum Cardinality $\ba$-Mediated Graph Problem} ($\ba$-MCMGP, for short).

Given $\ba \in \Lambda_d$, the goal of the $\ba$-MCMGP is to construct an $\ba$-mediated graph with the minimum number of nodes. The $\ba$-MCMGP has the flavor of a Continuous Location problem, where a set of services (the mediated elements) is to be located in a given space, taking into account a given set of \textit{users} (the elements in $\A_\ba$) located at certain coordinates in the space.  The interested reader is referred to \cite{blanco2014revisiting,blanco2016continuous,nickelpuerto} for further details on this family of problems.

We would like first to highlight that the $\ba$-MCMGP may not be unique, that is, it may exist different $\ba$-mediated graphs with minimum cardinality. A counterexample is described in the following example.

\begin{example}
    Let us consider $\ba = \frac{1}{74} (13, 17, 44) \in \Lambda_3$. In Figure \ref{fig:counter} we plot three different minimal $\ba$-mediated graphs with the same cardinality of their mediated sets ($7$). The three mediated graphs provide different extended representations of the cone with the same cardinality. Specifically:
    
    \begin{center}
    \begin{tabular}{p{4cm}p{4cm}p{4cm}}
    {$\!\begin{aligned}
        x^2 &\leq w_1z_3\\
        w_1^2 &\leq w_2w_3\\
        w_2^2 &\leq w_4 w_5\\
        w_3^2 &\leq w_5 z_2\\
        w_4^2 &\leq w_3 w_6\\
        w_5^2 &\leq x w_6\\
        w_6^2 &\leq w_4 z_1 \end{aligned}$} & {$\!\begin{aligned}
        x^2 &\leq w_1z_3\\
        w_1^2 &\leq w_2w_3\\
        w_2^2 &\leq x w_4\\
        w_3^2 &\leq w_5 z_2\\
        w_4^2 &\leq w_5 z_1\\
        w_5^2 &\leq w_4 w_6\\
        w_6^2 &\leq w_2 z_2 \end{aligned}$} & {$\!\begin{aligned}
        x^2 &\leq w_1z_3\\
        w_1^2 &\leq w_2w_3\\
        w_2^2 &\leq x w_4\\
        w_3^2 &\leq z_1 z_2\\
        w_4^2 &\leq w_1 w_5\\
        w_5^2 &\leq w_6 z_2\\
        w_6^2 &\leq w_2 w_4 \end{aligned}$}
    \end{tabular}
	\end{center}

	Observe that the three minimal representations are different and non-equivalent, in the sense that there are no possible permutations in the indices of the $w$-variables that transform one representation into another.
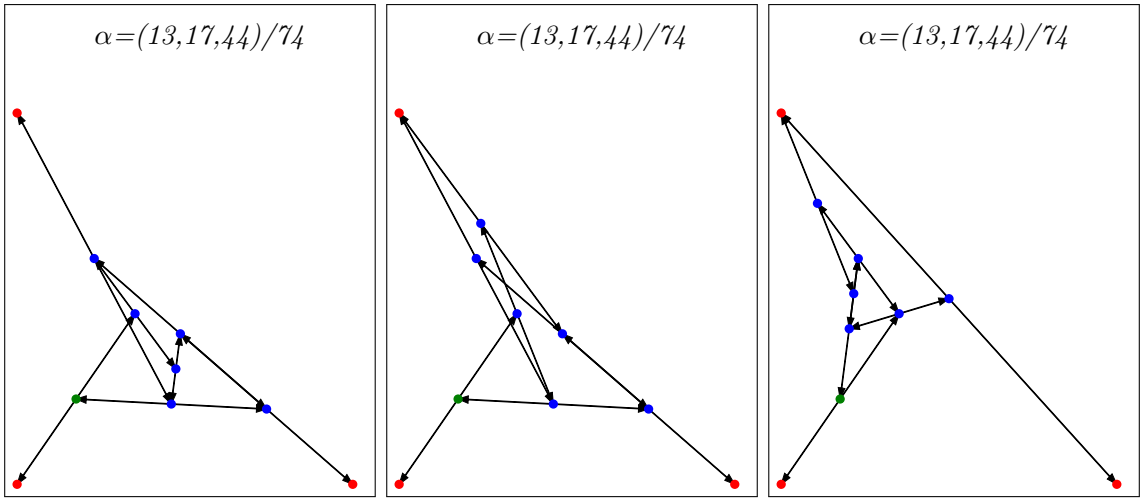
\begin{figure}[h]
\fbox{\adjustbox{width=0.31\textwidth}{
\begin{tikzpicture}

\definecolor{darkgray176}{RGB}{176,176,176}
\definecolor{green}{RGB}{0,128,0}

\begin{axis}[
hide x axis,
hide y axis,
tick align=outside,
tick pos=left,
title={\(\displaystyle \mathbf{\alpha}\)=(13,17,44)/74},
x grid style={darkgray176},
xmin=-15.0900432900433, xmax=95.4900432900433,
xtick style={color=black},
y grid style={darkgray176},
ymin=-1, ymax=81.4,
ytick style={color=black}
]
\path [draw=black, fill=black]
(axis cs:0,0)
--(axis cs:0.760714442404044,2.21298746881177)
--(axis cs:1.28896240133388,1.80903314727719)
--(axis cs:12.9404231625267,17.0455587580678)
--(axis cs:13.0595768374733,16.9544412419322)
--(axis cs:1.40811607628046,1.71791563114157)
--(axis cs:1.93636403521029,1.31396130960699)
--cycle;
\path [draw=black, fill=black]
(axis cs:26,34)
--(axis cs:25.239285557596,31.7870125311882)
--(axis cs:24.7110375986661,32.1909668527228)
--(axis cs:13.0595768374733,16.9544412419322)
--(axis cs:12.9404231625267,17.0455587580678)
--(axis cs:24.5918839237195,32.2820843688584)
--(axis cs:24.0636359647897,32.686038690393)
--cycle;
\path [draw=black, fill=black]
(axis cs:17.0000003178914,45)
--(axis cs:18.9785165920415,43.7504107263091)
--(axis cs:18.4638349807816,43.3293076046976)
--(axis cs:26.0580467982624,34.0474928332644)
--(axis cs:25.9419532017376,33.9525071667356)
--(axis cs:18.3477413842568,43.2343219381687)
--(axis cs:17.8330597729969,42.8132188165572)
--cycle;
\path [draw=black, fill=black]
(axis cs:35.0000015894572,23)
--(axis cs:33.0214851855179,24.2495890681913)
--(axis cs:33.5361667530397,24.6706922432605)
--(axis cs:25.9419532066705,33.9525071607065)
--(axis cs:26.0580467933295,34.0474928392935)
--(axis cs:33.6522603396988,24.7656779218475)
--(axis cs:34.1669419072206,25.1867810969166)
--cycle;
\path [draw=black, fill=black]
(axis cs:74,0)
--(axis cs:71.7990235128369,0.794797145789423)
--(axis cs:72.2110869817306,1.31674417145706)
--(axis cs:54.9535279480063,14.9411337940976)
--(axis cs:55.0464745951252,15.0588662059024)
--(axis cs:72.3040336288494,1.43447658326179)
--(axis cs:72.7160970977431,1.95642360892942)
--cycle;
\path [draw=black, fill=black]
(axis cs:36.0000025431315,30)
--(axis cs:38.2009790302947,29.2052028542106)
--(axis cs:37.788915561401,28.6832558285429)
--(axis cs:55.0464745951252,15.0588662059024)
--(axis cs:54.9535279480063,14.9411337940976)
--(axis cs:37.6959689142821,28.5655234167382)
--(axis cs:37.2839054453884,28.0435763910706)
--cycle;
\path [draw=black, fill=black]
(axis cs:0,74)
--(axis cs:1.76109379980539,72.4590429505379)
--(axis cs:1.18739965243411,72.1227394785488)
--(axis cs:17.0647026653393,45.0379289630063)
--(axis cs:16.9352979704436,44.9620710369937)
--(axis cs:1.05799495753833,72.0468815525362)
--(axis cs:0.484300810167042,71.7105780805471)
--cycle;
\path [draw=black, fill=black]
(axis cs:34.0000006357829,16)
--(axis cs:32.2389068359775,17.5409570494621)
--(axis cs:32.8126009833488,17.8772605214512)
--(axis cs:16.9352979704436,44.9620710369937)
--(axis cs:17.0647026653393,45.0379289630063)
--(axis cs:32.9420056782446,17.9531184474638)
--(axis cs:33.5156998256159,18.2894219194529)
--cycle;
\path [draw=black, fill=black]
(axis cs:34.0000006357829,16)
--(axis cs:33.581393728716,18.3023397354335)
--(axis cs:34.2397101294442,18.2082944456409)
--(axis cs:34.9257553788488,23.0106066116307)
--(axis cs:35.0742478000657,22.9893933883693)
--(axis cs:34.3882025506612,18.1870812223795)
--(axis cs:35.0465189513894,18.0930359325869)
--cycle;
\path [draw=black, fill=black]
(axis cs:36.0000025431315,30)
--(axis cs:36.4186094501985,27.6976602645665)
--(axis cs:35.7602930494702,27.7917055543591)
--(axis cs:35.0742478000657,22.9893933883693)
--(axis cs:34.9257553788488,23.0106066116307)
--(axis cs:35.6118006282533,27.8129187776205)
--(axis cs:34.953484227525,27.9069640674131)
--cycle;
\path [draw=black, fill=black]
(axis cs:13,17)
--(axis cs:15.2526854754043,17.6335677934541)
--(axis cs:15.2210546519985,16.9693204818218)
--(axis cs:34.0035680218813,16.0749151103345)
--(axis cs:33.9964332496845,15.9250848896655)
--(axis cs:15.2139198798017,16.8194902611529)
--(axis cs:15.1822890563959,16.1552429495207)
--cycle;
\path [draw=black, fill=black]
(axis cs:55.0000012715658,15)
--(axis cs:52.7473157961614,14.3664322065459)
--(axis cs:52.7789466195672,15.0306795181781)
--(axis cs:33.9964332496845,15.9250848896655)
--(axis cs:34.0035680218813,16.0749151103345)
--(axis cs:52.786081391764,15.1805097388471)
--(axis cs:52.8177122151698,15.8447570504793)
--cycle;
\path [draw=black, fill=black]
(axis cs:55.0000012715658,15)
--(axis cs:52.7990247844026,15.7947971457894)
--(axis cs:53.2110882532963,16.3167441714571)
--(axis cs:35.9535292195721,29.9411337940976)
--(axis cs:36.046475866691,30.0588662059024)
--(axis cs:53.3040349004152,16.4344765832618)
--(axis cs:53.7160983693089,16.9564236089294)
--cycle;
\path [draw=black, fill=black]
(axis cs:17.0000003178914,45)
--(axis cs:19.2009768761958,44.2052030512171)
--(axis cs:18.7889134540209,43.6832559886662)
--(axis cs:36.0464758614219,30.0588662100621)
--(axis cs:35.9535292248411,29.9411337899379)
--(axis cs:18.6959668174401,43.5655235685419)
--(axis cs:18.2839033952652,43.043576505991)
--cycle;
\addplot [semithick, red, mark=*, mark size=1.5, mark options={solid}, only marks]
table {%
74 0
};
\addplot [semithick, red, mark=*, mark size=1.5, mark options={solid}, only marks]
table {%
0 74
};
\addplot [semithick, red, mark=*, mark size=1.5, mark options={solid}, only marks]
table {%
0 0
};
\addplot [semithick, green, mark=*, mark size=1.5, mark options={solid}, only marks]
table {%
13 17
};
\addplot [semithick, blue, mark=*, mark size=1.5, mark options={solid}, only marks]
table {%
26 34
};
\addplot [semithick, blue, mark=*, mark size=1.5, mark options={solid}, only marks]
table {%
55.0000012715658 15
};
\addplot [semithick, blue, mark=*, mark size=1.5, mark options={solid}, only marks]
table {%
17.0000003178914 45
};
\addplot [semithick, blue, mark=*, mark size=1.5, mark options={solid}, only marks]
table {%
35.0000015894572 23
};
\addplot [semithick, blue, mark=*, mark size=1.5, mark options={solid}, only marks]
table {%
34.0000006357829 16
};
\addplot [semithick, blue, mark=*, mark size=1.5, mark options={solid}, only marks]
table {%
36.0000025431315 30
};
\end{axis}

\end{tikzpicture}}}~\fbox{\adjustbox{width=0.31\textwidth}{
\begin{tikzpicture}

\definecolor{darkgray176}{RGB}{176,176,176}
\definecolor{green}{RGB}{0,128,0}

\begin{axis}[
hide x axis,
hide y axis,
tick align=outside,
tick pos=left,
title={\(\displaystyle \mathbf{\alpha}\)=(13,17,44)/74},
x grid style={darkgray176},
xmin=-15.0900432900433, xmax=95.4900432900433,
xtick style={color=black},
y grid style={darkgray176},
ymin=-1, ymax=81.4,
ytick style={color=black}
]
\path [draw=black, fill=black]
(axis cs:0,0)
--(axis cs:0.760714442404044,2.21298746881177)
--(axis cs:1.28896240133388,1.80903314727719)
--(axis cs:12.9404231625267,17.0455587580678)
--(axis cs:13.0595768374733,16.9544412419322)
--(axis cs:1.40811607628046,1.71791563114157)
--(axis cs:1.93636403521029,1.31396130960699)
--cycle;
\path [draw=black, fill=black]
(axis cs:26,34)
--(axis cs:25.239285557596,31.7870125311882)
--(axis cs:24.7110375986661,32.1909668527228)
--(axis cs:13.0595768374733,16.9544412419322)
--(axis cs:12.9404231625267,17.0455587580678)
--(axis cs:24.5918839237195,32.2820843688584)
--(axis cs:24.0636359647897,32.686038690393)
--cycle;
\path [draw=black, fill=black]
(axis cs:0,74)
--(axis cs:1.76109378723384,72.4590429361704)
--(axis cs:1.1873996371189,72.1227394688616)
--(axis cs:17.0647023477573,45.0379289624784)
--(axis cs:16.9352976522427,44.9620710375216)
--(axis cs:1.05799494160425,72.0468815439048)
--(axis cs:0.484300791489306,71.710578076596)
--cycle;
\path [draw=black, fill=black]
(axis cs:34,16)
--(axis cs:32.2389062127662,17.5409570638296)
--(axis cs:32.8126003628811,17.8772605311384)
--(axis cs:16.9352976522427,44.9620710375216)
--(axis cs:17.0647023477573,45.0379289624784)
--(axis cs:32.9420050583958,17.9531184560952)
--(axis cs:33.5156992085107,18.289421923404)
--cycle;
\path [draw=black, fill=black]
(axis cs:0,74)
--(axis cs:1.97851629578158,72.750410760559)
--(axis cs:1.46383469181137,72.3293076300379)
--(axis cs:18.0580467974402,52.0474928342693)
--(axis cs:17.9419532025598,51.9525071657307)
--(axis cs:1.34774109693087,72.2343219614993)
--(axis cs:0.833059492960665,71.8132188309783)
--cycle;
\path [draw=black, fill=black]
(axis cs:36,30)
--(axis cs:34.0214837042184,31.249589239441)
--(axis cs:34.5361653081886,31.6706923699621)
--(axis cs:17.9419532025598,51.9525071657307)
--(axis cs:18.0580467974402,52.0474928342693)
--(axis cs:34.6522589030691,31.7656780385007)
--(axis cs:35.1669405070393,32.1867811690217)
--cycle;
\path [draw=black, fill=black]
(axis cs:18,52)
--(axis cs:19.5778479406176,50.2718808269426)
--(axis cs:18.9701632607852,50.001798747017)
--(axis cs:26.0685358661465,34.030460384954)
--(axis cs:25.9314641338535,33.969539615046)
--(axis cs:18.8330915284921,49.940877977109)
--(axis cs:18.2254068486597,49.6707958971835)
--cycle;
\path [draw=black, fill=black]
(axis cs:34,16)
--(axis cs:32.4221520593824,17.7281191730574)
--(axis cs:33.0298367392148,17.998201252983)
--(axis cs:25.9314641338535,33.969539615046)
--(axis cs:26.0685358661465,34.030460384954)
--(axis cs:33.1669084715079,18.059122022891)
--(axis cs:33.7745931513403,18.3292041028165)
--cycle;
\path [draw=black, fill=black]
(axis cs:13,17)
--(axis cs:15.2526854763157,17.6335677902138)
--(axis cs:15.2210546519544,16.969320478627)
--(axis cs:34.0035673862062,16.0749151103293)
--(axis cs:33.9964326137938,15.9250848896707)
--(axis cs:15.2139198795421,16.8194902579684)
--(axis cs:15.1822890551808,16.1552429463816)
--cycle;
\path [draw=black, fill=black]
(axis cs:55,15)
--(axis cs:52.7473145236843,14.3664322097862)
--(axis cs:52.7789453480456,15.030679521373)
--(axis cs:33.9964326137938,15.9250848896707)
--(axis cs:34.0035673862062,16.0749151103293)
--(axis cs:52.7860801204579,15.1805097420316)
--(axis cs:52.8177109448192,15.8447570536184)
--cycle;
\path [draw=black, fill=black]
(axis cs:17,45)
--(axis cs:19.2009765130327,44.2052029258493)
--(axis cs:18.7889130611276,43.6832558867696)
--(axis cs:36.0464733216434,30.058866207415)
--(axis cs:35.9535266783566,29.941133792585)
--(axis cs:18.6959664178408,43.5655234719396)
--(axis cs:18.2839029659357,43.0435764328598)
--cycle;
\path [draw=black, fill=black]
(axis cs:55,15)
--(axis cs:52.7990234869673,15.7947970741507)
--(axis cs:53.2110869388724,16.3167441132304)
--(axis cs:35.9535266783566,29.941133792585)
--(axis cs:36.0464733216434,30.058866207415)
--(axis cs:53.3040335821592,16.4344765280604)
--(axis cs:53.7160970340643,16.9564235671402)
--cycle;
\path [draw=black, fill=black]
(axis cs:74,0)
--(axis cs:71.7990234869673,0.794797074150691)
--(axis cs:72.2110869388724,1.31674411323042)
--(axis cs:54.9535266783566,14.941133792585)
--(axis cs:55.0464733216434,15.058866207415)
--(axis cs:72.3040335821592,1.43447652806043)
--(axis cs:72.7160970340643,1.95642356714016)
--cycle;
\path [draw=black, fill=black]
(axis cs:36,30)
--(axis cs:38.2009765130327,29.2052029258493)
--(axis cs:37.7889130611276,28.6832558867696)
--(axis cs:55.0464733216434,15.058866207415)
--(axis cs:54.9535266783566,14.941133792585)
--(axis cs:37.6959664178408,28.5655234719396)
--(axis cs:37.2839029659357,28.0435764328598)
--cycle;
\addplot [semithick, red, mark=*, mark size=1.5, mark options={solid}, only marks]
table {%
74 0
};
\addplot [semithick, red, mark=*, mark size=1.5, mark options={solid}, only marks]
table {%
0 74
};
\addplot [semithick, red, mark=*, mark size=1.5, mark options={solid}, only marks]
table {%
0 0
};
\addplot [semithick, green, mark=*, mark size=1.5, mark options={solid}, only marks]
table {%
13 17
};
\addplot [semithick, blue, mark=*, mark size=1.5, mark options={solid}, only marks]
table {%
17 45
};
\addplot [semithick, blue, mark=*, mark size=1.5, mark options={solid}, only marks]
table {%
18 52
};
\addplot [semithick, blue, mark=*, mark size=1.5, mark options={solid}, only marks]
table {%
26 34
};
\addplot [semithick, blue, mark=*, mark size=1.5, mark options={solid}, only marks]
table {%
34 16
};
\addplot [semithick, blue, mark=*, mark size=1.5, mark options={solid}, only marks]
table {%
36 30
};
\addplot [semithick, blue, mark=*, mark size=1.5, mark options={solid}, only marks]
table {%
55 15
};
\end{axis}

\end{tikzpicture}}}~\fbox{\adjustbox{width=0.31\textwidth}{
\begin{tikzpicture}

\definecolor{darkgray176}{RGB}{176,176,176}
\definecolor{green}{RGB}{0,128,0}

\begin{axis}[
hide x axis,
hide y axis,
tick align=outside,
tick pos=left,
title={\(\displaystyle \mathbf{\alpha}\)=(13,17,44)/74},
x grid style={darkgray176},
xmin=-15.0900432900433, xmax=95.4900432900433,
xtick style={color=black},
y grid style={darkgray176},
ymin=-1, ymax=81.4,
ytick style={color=black}
]
\path [draw=black, fill=black]
(axis cs:0,0)
--(axis cs:0.760714442404044,2.21298746881177)
--(axis cs:1.28896240133388,1.80903314727719)
--(axis cs:12.9404231625267,17.0455587580678)
--(axis cs:13.0595768374733,16.9544412419322)
--(axis cs:1.40811607628046,1.71791563114157)
--(axis cs:1.93636403521029,1.31396130960699)
--cycle;
\path [draw=black, fill=black]
(axis cs:26,34)
--(axis cs:25.239285557596,31.7870125311882)
--(axis cs:24.7110375986661,32.1909668527228)
--(axis cs:13.0595768374733,16.9544412419322)
--(axis cs:12.9404231625267,17.0455587580678)
--(axis cs:24.5918839237195,32.2820843688584)
--(axis cs:24.0636359647897,32.686038690393)
--cycle;
\path [draw=black, fill=black]
(axis cs:0,74)
--(axis cs:1.57784794061764,72.2718808269426)
--(axis cs:0.970163260785173,72.001798747017)
--(axis cs:8.06853586614652,56.030460384954)
--(axis cs:7.93146413385348,55.969539615046)
--(axis cs:0.833091528492134,71.940877977109)
--(axis cs:0.225406848659663,71.6707958971835)
--cycle;
\path [draw=black, fill=black]
(axis cs:16,38)
--(axis cs:14.4221520593824,39.7281191730574)
--(axis cs:15.0298367392148,39.998201252983)
--(axis cs:7.93146413385348,55.969539615046)
--(axis cs:8.06853586614652,56.030460384954)
--(axis cs:15.1669084715079,40.059122022891)
--(axis cs:15.7745931513403,40.3292041028165)
--cycle;
\path [draw=black, fill=black]
(axis cs:8,56)
--(axis cs:9.97851629578158,54.750410760559)
--(axis cs:9.46383469181137,54.3293076300379)
--(axis cs:17.0580467974402,45.0474928342693)
--(axis cs:16.9419532025598,44.9525071657307)
--(axis cs:9.34774109693087,54.2343219614993)
--(axis cs:8.83305949296066,53.8132188309783)
--cycle;
\path [draw=black, fill=black]
(axis cs:26,34)
--(axis cs:24.0214837042184,35.249589239441)
--(axis cs:24.5361653081886,35.6706923699621)
--(axis cs:16.9419532025598,44.9525071657307)
--(axis cs:17.0580467974402,45.0474928342693)
--(axis cs:24.6522589030691,35.7656780385007)
--(axis cs:25.1669405070393,36.1867811690217)
--cycle;
\path [draw=black, fill=black]
(axis cs:17,45)
--(axis cs:17.4186072144624,42.6976603204566)
--(axis cs:16.7602908011778,42.7917055223544)
--(axis cs:16.0742462120246,37.9893933982822)
--(axis cs:15.9257537879754,38.0106066017178)
--(axis cs:16.6117983771286,42.81291872579)
--(axis cs:15.9534819638439,42.9069639276878)
--cycle;
\path [draw=black, fill=black]
(axis cs:15,31)
--(axis cs:14.5813927855376,33.3023396795434)
--(axis cs:15.2397091988222,33.2082944776456)
--(axis cs:15.9257537879754,38.0106066017178)
--(axis cs:16.0742462120246,37.9893933982822)
--(axis cs:15.3882016228714,33.18708127421)
--(axis cs:16.0465180361561,33.0930360723122)
--cycle;
\path [draw=black, fill=black]
(axis cs:74,0)
--(axis cs:71.9069639276878,1.04651803615609)
--(axis cs:72.3771899371769,1.51674404564514)
--(axis cs:36.946966991411,36.946966991411)
--(axis cs:37.053033008589,37.053033008589)
--(axis cs:72.4832559543549,1.62281006282313)
--(axis cs:72.9534819638439,2.09303607231218)
--cycle;
\path [draw=black, fill=black]
(axis cs:0,74)
--(axis cs:2.09303607231218,72.9534819638439)
--(axis cs:1.62281006282313,72.4832559543549)
--(axis cs:37.053033008589,37.053033008589)
--(axis cs:36.946966991411,36.946966991411)
--(axis cs:1.51674404564514,72.3771899371769)
--(axis cs:1.04651803615609,71.9069639276878)
--cycle;
\path [draw=black, fill=black]
(axis cs:37,37)
--(axis cs:35.0529311971384,35.7019541314256)
--(axis cs:34.8779581222866,36.3435220725487)
--(axis cs:26.0197338054344,33.9276427134072)
--(axis cs:25.9802661945656,34.0723572865928)
--(axis cs:34.8384905114178,36.4882366457343)
--(axis cs:34.6635174365661,37.1298045868574)
--cycle;
\path [draw=black, fill=black]
(axis cs:15,31)
--(axis cs:16.9470688028616,32.2980458685744)
--(axis cs:17.1220418777134,31.6564779274513)
--(axis cs:25.9802661945656,34.0723572865928)
--(axis cs:26.0197338054344,33.9276427134072)
--(axis cs:17.1615094885822,31.5117633542656)
--(axis cs:17.3364825634339,30.8701954131426)
--cycle;
\path [draw=black, fill=black]
(axis cs:13,17)
--(axis cs:12.5813927855376,19.3023396795434)
--(axis cs:13.2397091988222,19.2082944776456)
--(axis cs:14.9257537879754,31.0106066017178)
--(axis cs:15.0742462120246,30.9893933982822)
--(axis cs:13.3882016228714,19.18708127421)
--(axis cs:14.0465180361561,19.0930360723122)
--cycle;
\path [draw=black, fill=black]
(axis cs:17,45)
--(axis cs:17.4186072144624,42.6976603204566)
--(axis cs:16.7602908011778,42.7917055223544)
--(axis cs:15.0742462120246,30.9893933982822)
--(axis cs:14.9257537879754,31.0106066017178)
--(axis cs:16.6117983771286,42.81291872579)
--(axis cs:15.9534819638439,42.9069639276878)
--cycle;
\addplot [semithick, red, mark=*, mark size=1.5, mark options={solid}, only marks]
table {%
74 0
};
\addplot [semithick, red, mark=*, mark size=1.5, mark options={solid}, only marks]
table {%
0 74
};
\addplot [semithick, red, mark=*, mark size=1.5, mark options={solid}, only marks]
table {%
0 0
};
\addplot [semithick, green, mark=*, mark size=1.5, mark options={solid}, only marks]
table {%
13 17
};
\addplot [semithick, blue, mark=*, mark size=1.5, mark options={solid}, only marks]
table {%
8 56
};
\addplot [semithick, blue, mark=*, mark size=1.5, mark options={solid}, only marks]
table {%
17 45
};
\addplot [semithick, blue, mark=*, mark size=1.5, mark options={solid}, only marks]
table {%
16 38
};
\addplot [semithick, blue, mark=*, mark size=1.5, mark options={solid}, only marks]
table {%
37 37
};
\addplot [semithick, blue, mark=*, mark size=1.5, mark options={solid}, only marks]
table {%
26 34
};
\addplot [semithick, blue, mark=*, mark size=1.5, mark options={solid}, only marks]
table {%
15 31
};
\end{axis}

\end{tikzpicture}}}
    \caption{Counterexample on the uniqueness of minimal $\ba$-mediated graphs.\label{fig:counter}}
    \end{figure}
    
\end{example}

For the sake of describing the problem, we denote by $I_A = \{0, 1, \ldots, d\}$  the index sets for the elements in $\A_{\ba} \cup \{\b_\ba\}$, where $0$ is assumed to be the index associated to $\b_\ba$. Observe that the size of the mediated graph is unknown and to be determined by our model. Thus, unlike what happens in other location problems, the problem not only consists of \textit{filling} adequately the coordinates of a given number of vectors, but this number is to be decided. Then, we overestimate the size of the mediated graph by $d+\Delta_{\ba}$ (where $\Delta_{\ba}$ is the upper bound described in Remark \ref{bounds}), and denote by $I_M = \{d+1, \ldots, d+\Delta_{\ba}\}$ the index set for the \textit{enlarged} mediated graph. Once, all these $\Delta_{\ba}$ coordinates are decided, and in order to account for a minimal size set of coordinates, we will group in a single element each of the equal-coordinates vectors, transforming the set of $\Delta_{\ba}$ coordinates into one that formally represents the unknown nodes of the $\ba$-mediated graph. That is, if $\x_{d+1}, \ldots, \x_{d+\Delta_{\ba}}$ are the coordinates of the enlarged mediated set, we compute the effective nodes, $X$, by removing repeated elements from the multiset.

Additionally, for the sake of simplification, we denote by $I=I_A \cup I_M = \{0, \ldots, d+\Delta_{\ba}\}$, the whole sets of coordinates involved in the problem.

We use the following set of variables in our model:

\begin{itemize}
\item To assure the correct construction of the mediated graph $G=(\A_{\ba}\cup X, A)$, we consider the following binary variables that account for the active arcs in the graph:
$$
y_{ij} = \begin{cases}
1 & \mbox{if $(\x_i,\x_j) \in A$,}\\
0 & \mbox{otherwise},
\end{cases}
$$
for all $i\in I_M\cup\{0\}, j \in I$, $i\neq j$.
\item In order to account for the number of effective mediated elements in the overestimated set, we use the standard variables in facility location to activate nodes:
$$
z_j = \begin{cases}
1 & \mbox{if node $\x_j$ is an active node in the mediated graph,}\\
0 & \mbox{otherwise}
\end{cases}
$$
for all $j \in I$.

Note that the amount $\sum_{j\in I_M} z_j$ provides the number of different nodes in the $\ba$-mediated set, and then, its cardinality. Since all the elements in $\A_{\ba} \cup \{\b_\alpha\}$ must be nodes of the graph, we fix $z_0 = \cdots = z_d=1$ in our model.
\item We consider continuous variables in $[0,\hat{s}]$ to model the coordinates of the nodes in the mediated graph:
$$
x_{jr} \in [0,\hat{s}]: \text{$k$th coordinate of the $j$-th node in the mediated graph.}
$$
for $j\in I$, $k=1, \ldots, d-1$. The first $d$ nodes have fixed coordinates to the elements in $\A_{\ba} \cup \{\b_\alpha\}$ that must be part of the mediated graph. Thus, we fix their values to $\x_0 = \b_\ba$, $\x_j = \hat{s} \e_j$ (for $j=1, \ldots, d-1$), and $\x_d = \mathbf{0}$.
\end{itemize}

With the above variables, the mathematical optimization model that we propose to find a minimal $\ba$-mediated graph, and then, the minimal extended representation of the cone, is the following:
\begin{align}
\min & \sum_{j\in I_M \cup \{0\}} z_j\\
\mbox{s.t. } & \dsum_{j\in I} y_{ij} = 2 z_i, \forall i \in I_M \cup \{0\},\label{ctr:1}\\
& y_{ij} \leq z_j, \forall i \in I_M \cup\{0\}, j \in I, \label{ctr:2}\\
& \x_i \geq \frac{1}{2}(\x_j + \x_k) - \hat{s} (2- y_{ij}-y_{ik}), \forall i \in I_M \cup \{0\}, j\neq k \in I,\label{ctr:3a}\\
& \x_i \leq \frac{1}{2}(\x_j + \x_k) + \hat{s} (2- y_{ij}-y_{ik}), \forall i \in I_M \cup \{0\}, j\neq k \in I,\label{ctr:3b}\\
 &  \|\x_i - \x_j\|_1 \geq \varepsilon (z_i+z_j-1), \forall i,j \in I (i\neq j),\label{ctr:4}\\ 
 & \x_0 = \b_{\ba},\\
& \x_i=\hat{s}\e_i, \forall i\in I_A\setminus \{0,d\},\\
& \x_d=\0,\\
& z_0 = \cdots = z_d =1,\\
& y_{ij} \in \{0,1\}, \forall i\in I_M\cup\{0\}, j \in I, i\neq j,\label{ctr:14}\\
& z_j \in \{0,1\}, \forall j \in I,\label{ctr:16}\\
& \x_i \in \R^{d-1}, \forall i \in I_M.\label{ctr:17}
\end{align}

The objective function in the above problem accounts for the number of (active) nodes in the mediated set of the $\ba$-mediated graph. Constraints \eqref{ctr:1} ensures that there are exactly two outgoing arcs from each active node in the mediated set, with endnodes in $\A_\ba$ or in the new set of mediated nodes. Constraints \eqref{ctr:2} ensure that only active nodes can be chosen to construct arcs in the mediated graph. Constraints \eqref{ctr:3a} and \eqref{ctr:3b} are the requirements for the mediated graph, i.e., the arcs are constructed as averages of the endnode coordinates. With Constraints \eqref{ctr:4} we assure that in case two mediated are active, then their coordinates must be different (since their $\ell_1$ distance between them must be greater than $\varepsilon$, and zero otherwise). The rest of the constraints are the variable fixing equations and the domains for each of the variables.

The above formulation is a mixed integer linear programming model, which can be hard to solve for general instances of the problem since the number of binary variables to decide is $\Delta_{\ba} + (\Delta_{\ba}+d-1)\Delta_{\ba}$. We incorporate different valid inequalities that strengthen the linear relaxation and break symmetries of the problem.

Specifically, in case $\x_j$ is not an active node, we fix its coordinates to the arbitrary vector $(\hat{s}, \stackrel{d-1)}{\ldots}, \hat{s})$, by adding the following constraints:
\begin{equation}\label{ctr:vi1}
\hat{s} (1-z_j) \leq x_{jk} \leq \hat{s} (1+z_j), \forall j \in I, k=1, \ldots, d-1.
\end{equation}
Note that in case $z_j=0$, these constraints enforce that $\x_j = (\hat{s}, \stackrel{d-1)}{\ldots}, \hat{s})$, otherwise, the constraint is redundant ($0 \leq x_{jk} \leq 2\hat{s}$, which is always satisfied). Observe that in case the node is not activated, the assignment of coordinates to the node is not relevant, and any value is possible. Thus, fixing it to a given value allows to reduce the exploration of the branch-and-bound tree.

The second set of valid inequalities that we consider are the following:
\begin{equation}\label{ctr:vi2}
z_j \leq z_{j-1}, \forall j \in I\backslash\{0\}
\end{equation}
These constraints force the non-activated nodes to be assigned to the last indices in $I$. Thus, given $\kappa = \sum_{j \in I} z_j$, the nodes in the mediated graph have coordinates $\x_0, \ldots, \x_{d+\kappa}$, being the remainder, $\x_{d+\kappa+1}=  \cdots = \x_{d+\Delta_{\ba}} =  (\hat{s}, \stackrel{d-1)}{\ldots}, \hat{s})$ by constraints \eqref{ctr:vi1}.

Observe that there is still a lot of symmetry in the problem since any permutation of the coordinates of the active mediated nodes is an alternative solution with the same objective value. We break symmetries in our model by enforcing the active nodes to be sorted in non-decreasing order with the following valid inequalities:
\begin{equation}\label{ctr:vi3}
x_{jk} \leq x_{(j-1)k}, \forall j \in I_M \backslash\{d\}.
\end{equation}
where $k \in \{1, \ldots, d-1\}$ is a chosen coordinate to sort the elements in the mediated set.

\begin{example}
    Taking $\ba = \frac{1}{6} (1,2,3)$, the minimal cardinality $\ba$-mediated set that we obtained solving the optimization problem above is the one drawn in Figure \ref{fig:mediated}.  Note that this graph induces a minimal representation of the cone  $\Q^{1+3}(\ba)=\{x^6\leq z_1z_2^2z_3^3\}$ with complexity $3$ (the number of extra mediated nodes in the graph). The constraints induced by the graph are:
    \begin{eqnarray}
        x^2 &\leq w_1 z_3\label{rep:1}\\
        w_1^2 &\leq w_2 z_2\label{rep:2}\\
        w_2^2 &\leq w_1 z_1\label{rep:3}
    \end{eqnarray}
    One can easily check that the obtained representation is valid by combining the above constraints:
    \begin{eqnarray*}
    \begin{split}
     x^6 w_2 &\stackrel{\eqref{rep:1}}{\leq} w_1^3 z_3^3 w_2 \stackrel{\eqref{rep:2}}{\leq} z_2 z_3^3 w_1  w_2^2\\
     & \stackrel{\eqref{rep:3}}{\leq} z_2 z_3^3 w_1^2 z_1 \stackrel{\eqref{rep:2}}{\leq} z_1 z_2^2 z_3^3 w_2 \Rightarrow x^6 \leq z_1 z_2^2 z_3^3
    \end{split}
    \end{eqnarray*}
\end{example}

\begin{remark}
    The representations that have been already applied in the literature, specially in the representation of $p$-order cones, and that allow us to derive an upper bound for the cardinality of the optimal mediated graphs (Remark \ref{bounds}) can be \textit{easily} constructed using binary decomposition (see \cite{blanco2014revisiting} for further details). The special mediated graphs obtained in such are characterized by verifying that each mediated node is linked in the mediated graph with at most one other mediated node. This condition can be easily enforced in our model with the linear constraints:
    \begin{equation*}
    \sum_{i\in I_M} y_{ij} \leq z_j, \forall j \in I_M.
    \end{equation*}
\end{remark}

\section{Experiments}\label{sec:5}

In this section, we report the results of some experiments that we run to test our approach and determine its limitation in size. First, we derive some tests to analyze the computational complexity of the $\ba$-MCMGP. Then, we apply the minimal extended representation of some cones to a practical problem in facility location.

\subsection{Computational experiments for the $\ba$-MCMGP}

We have run a series of experiments to analyze the performance of solving $\ba$-MCMGP on randomly generated instances. Note that the only input for $\ba$-MCMGP is the vector $\ba \in \Lambda_d$. Thus, we randomly generate sequences of $d$ integers in $[1, q]$, with $d$ ranging in $\{2, 3,4\}$ and $q$ ranging in $\{10,20,30,40\}$. For each combination of $d$ and $q$, we generate $5$ instances. Each of these instances results in a vector $\ba \in \Lambda_d$ normalizing the resulting vector by the sum of its coordinates. In total, $60$ instances were generated.

For each of these instances, we run the model described in Section \ref{model}. The formulation has been coded in Python 3.8 in a Huawei FusionServer Pro XH321 (\texttt{albaic\'in} at Universidad de Granada - \url{https://supercomputacion.ugr.es/arquitecturas/albaicin/}) with an Intel Xeon Gold 6258R CPU @ 2.70GHz with 28 cores. We used Gurobi 10.0.3 as an optimization solver. A time limit of 5 hours was fixed for all the instances.

In the results of our preliminary experiments, we observed that in most of the instances, the minimal cardinality of the mediated graph was reached in the lower bound. Thus, instead of fixing $\Delta_{\ba}$ as the upper bound, which results in an optimization problem with an overestimated number of variables, we proceed by fixing, initially $\Delta_{\ba}$ to the value of the lower bound, and in case the problem is infeasible, we increase by one unit this bound, until the problem is feasible. This strategy dramatically reduces the size of the optimization problem. 

The set of instances used in our experiments is available at the GitHub repository \url{https://github.com/vblancoOR/extended-representations}. 

In Table \ref{table:mediated} we summarize the results of our experiments. For each value of $d$ (dimension of $\ba$) and $q$ (maximum value of $s$), we report in the table the minimum, average, and maximum CPU time (in seconds) required by Gurobi to solve the problem (columns bellow  \texttt{CPU Time}); the number of nodes in the branch and bound approach for solving the feasible problem (\texttt{\#Nodes}); the deviations of the optimal value with respect to both the lower bound (\texttt{DevLB}) and the upper bound (\texttt{DevUB}). Finally, in column \texttt{Opt} we report the average optimal value for the problem, i.e., the cardinality of the minimal $\ba$-mediated set. For $d=4$ we highlight with an asterisk ($*$) in the Max CPU time column that for one of the instances summarized in that row reached the time limit in one of the runs, and in that case, the minimum cardinality of the $\ba$-mediated graph might be smaller but Gurobi was not able to find a feasible solution within the time limit and then, the problem for a larger value of $\Delta_\ba$ was successfully solved.


\begin{table}
\centering
\begin{tabular}{|l|l|l|l|l|l|l|l|l|}\cline{3-5}
\multicolumn{2}{c|}{} & \multicolumn{3}{c|}{\texttt{CPU Time (secs)}} & \multicolumn{4}{c}{}\\\hline
$d$ & $q$ & \texttt{Min} & \texttt{Av} & \texttt{Max} & \texttt{\#Nodes} & \texttt{DevLB} & \texttt{DevUB} & \texttt{Opt} \\
\hline
\multirow{4}{*}{$2$}  & 10 & 0.02 & 0.04 & 0.04 & 0 & 0.00\% & 19.67\% & 3.8 \\
 & 20 & 0.06 & 0.16 & 0.35 & 1028 & 0.00\% & 27.88\% & 5.2 \\
 & 30 & 0.06 & 0.24 & 0.61 & 2659 & 0.00\% & 20.48\% & 5.4 \\
 & 40 & 0.15 & 0.3 & 0.41 & 2901 & 0.00\% & 19.52\% & 5.8 \\
\hline
\multirow{4}{*}{$3$}  & 10 & 0.13 & 0.43 & 0.96 & 1415 & 0.00\% & 26.93\% & 4.6 \\
 & 20 & 0.89 & 69.67 & 179.97 & 6101 & 12.38\% & 26.49\% & 6.2 \\
 & 30 & 1.28 & 2.11 & 3.67 & 19319 & 0.00\% & 35.67\% & 6 \\
 & 40 & 1.38 & 58.2 & 281.54 & 73464 & 2.86\% & 29.13\% & 6.4 \\\hline
 \multirow{4}{*}{$4$}  & 10 & 0.69 & 344.76 & 850.48 & 2884667 & 18.10\% & 20.08\% & 6.2 \\
 & 20 & 397.14 & 8551.05 & 18000$^*$ & 63662316 & 22.14\% & 22.94\% & 7.2 \\
 & 30 & 412.06 & 6343 & 18000$^*$ & 4063237 & 15.87\% & 30.00\% & 7.6 \\
 & 40 & 2.85 & 815.54 & 1628.23 & 2542017 & 7.14\% & 15.00\% & 6.5 \\
\hline
\end{tabular} \caption{Summary of our computational experiments for the $\ba$-MCMGP.\label{table:mediated}}
\end{table}

Several observations can be drawn from the obtained results. First, that the $\ba$-MCMGP is a computationally challenging problem. As the dimension of $\ba$, $d$, increases, solving the $\ba$-MCMGP becomes more computationally costly. Additionally, as already proved in \cite{wang2022weighted}, in the planar case ($d=2$), the optimal cardinality of the mediated set in the mediated graph coincides with the lower bound ($\max\{d-1, \lceil\log_2(\hat{s})\rceil\}$). This is no longer the case for $d\geq 3$, but still in $75\%$ of the instances for $d=3$ the lower bound coincides with the optimal value. For $d=4$, the situation becomes different, since in only $20\%$ of the instances the solver was able to check that the lower bound is the optimal cardinality of the mediated set. It can be the case that although the lower bound is reached, the solver was not able to prove it within the time limit.

Finally, one can also observe that the integer linear problem is challenging and requires branching, being the number of nodes in the branch-and-bound required to certify optimally really large for some of the instances. 

The most interesting observation, that validates the study in this paper, is that the deviations with respect to a myopic representation of the cones (as a tree-based mediated graph) are large (see \texttt{DevUB}, i.e., the savings in the number of mediated nodes, or equivalently, in the number of variables required to represent the cones as second order cones. 

Note that although solving the $\ba$-MCMGP might require a significant amount of CPU time, the problems requiring the representation of a given power cone would need to represent the cone several times (as in most continuous location problems) but the $\ba$-MCMGP is solved just once. It is advisable then to solve optimization problems restricted to these cones, since the minimal representation results in a reduction in the CPU time required to solve them, as can be observed in the following section.

In figures \ref{plots:d3} and \ref{plots:d4} we show some of the optimal mediated graphs obtained in our experiments for the cases $d=3$ (in $\R^2$) and $d=4$ (in $\R^3$), respectively. In the plots, we represent with red dots the elements in $\A_{\ba}$, with a green dot the point $\b_{\ba}$, and with blue dots the rest of the elements in the mediated set. The arrows represent the arcs of the mediated graph. An intriguing question, that will be the goal of our further research, is to determine the topological structure of the mediated graphs and to exploit it into the above mathematical optimization problem.

 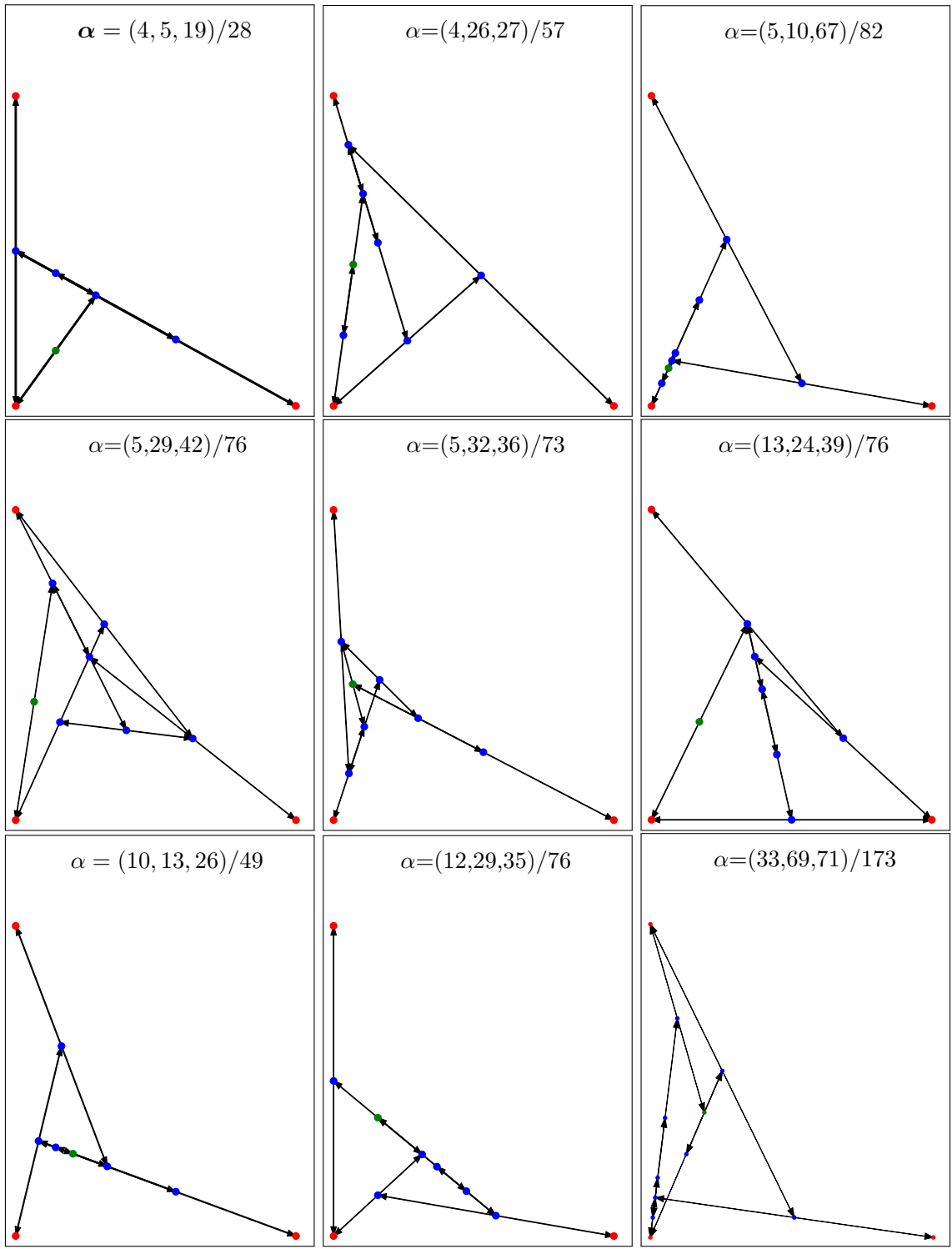
\begin{figure}[h]
     \fbox{\adjustbox{width=0.31\textwidth}{
\begin{tikzpicture}

\definecolor{darkgray176}{RGB}{176,176,176}
\definecolor{green}{RGB}{0,128,0}

\begin{axis}[
hide x axis,
hide y axis,
tick align=outside,
tick pos=left,
title={$\mathbf{\ba}=(4,5,19)/28$},
x grid style={darkgray176},
xmin=-6.43766233766234, xmax=36.2376623376623,
xtick style={color=black},
y grid style={darkgray176},
ymin=-1, ymax=30.8,
ytick style={color=black}
]
\path [draw=black, fill=black]
(axis cs:0,0)
--(axis cs:0.306100573301668,0.830844413247384)
--(axis cs:0.466178679237489,0.702781928498727)
--(axis cs:3.94143483929177,5.04685212856658)
--(axis cs:4.05856516070823,4.95314787143342)
--(axis cs:0.583309000653944,0.609077671365564)
--(axis cs:0.743387106589765,0.481015186616907)
--cycle;
\path [draw=black, fill=black]
(axis cs:8.00000000000001,9.99999999999999)
--(axis cs:7.69389942669834,9.16915558675261)
--(axis cs:7.53382132076252,9.29721807150126)
--(axis cs:4.05856516070823,4.95314787143342)
--(axis cs:3.94143483929177,5.04685212856658)
--(axis cs:7.41669099934606,9.39092232863443)
--(axis cs:7.25661289341024,9.51898481338308)
--cycle;
\path [draw=black, fill=black]
(axis cs:16,5.99999999999997)
--(axis cs:15.1234613528201,6.12521980673996)
--(axis cs:15.2151401398976,6.30857738089494)
--(axis cs:7.96645898033751,9.932917960675)
--(axis cs:8.03354101966251,10.067082039325)
--(axis cs:15.2822221792226,6.44274145954493)
--(axis cs:15.3739009663001,6.62609903369991)
--cycle;
\path [draw=black, fill=black]
(axis cs:0,14)
--(axis cs:0.876538647179918,13.87478019326)
--(axis cs:0.784859860102426,13.691422619105)
--(axis cs:8.03354101966251,10.067082039325)
--(axis cs:7.96645898033751,9.932917960675)
--(axis cs:0.717777820777432,13.557258540455)
--(axis cs:0.626099033699941,13.3739009663001)
--cycle;
\path [draw=black, fill=black]
(axis cs:28,0)
--(axis cs:27.1234613528201,0.125219806739987)
--(axis cs:27.2151401398976,0.30857738089497)
--(axis cs:15.9664589803375,5.93291796067498)
--(axis cs:16.0335410196625,6.06708203932497)
--(axis cs:27.2822221792226,0.442741459544957)
--(axis cs:27.3739009663001,0.62609903369994)
--cycle;
\path [draw=black, fill=black]
(axis cs:4.00000000000003,11.9999999999999)
--(axis cs:4.87653864717995,11.87478019326)
--(axis cs:4.78485986010246,11.691422619105)
--(axis cs:16.0335410196625,6.06708203932497)
--(axis cs:15.9664589803375,5.93291796067498)
--(axis cs:4.71777782077746,11.557258540455)
--(axis cs:4.62609903369997,11.3739009663)
--cycle;
\path [draw=black, fill=black]
(axis cs:0,28)
--(axis cs:0.28,27.16)
--(axis cs:0.075,27.16)
--(axis cs:0.075,14)
--(axis cs:-0.075,14)
--(axis cs:-0.075,27.16)
--(axis cs:-0.28,27.16)
--cycle;
\path [draw=black, fill=black]
(axis cs:0,0)
--(axis cs:-0.28,0.84)
--(axis cs:-0.075,0.84)
--(axis cs:-0.075,14)
--(axis cs:0.075,14)
--(axis cs:0.075,0.84)
--(axis cs:0.28,0.84)
--cycle;
\path [draw=black, fill=black]
(axis cs:8.00000000000001,9.99999999999999)
--(axis cs:7.12346135282009,10.12521980674)
--(axis cs:7.21514013989758,10.308577380895)
--(axis cs:3.96645898033753,11.9329179606749)
--(axis cs:4.03354101966253,12.0670820393249)
--(axis cs:7.28222217922257,10.4427414595449)
--(axis cs:7.37390096630006,10.6260990336999)
--cycle;
\path [draw=black, fill=black]
(axis cs:0,14)
--(axis cs:0.876538647179917,13.87478019326)
--(axis cs:0.784859860102424,13.691422619105)
--(axis cs:4.03354101966253,12.0670820393249)
--(axis cs:3.96645898033753,11.9329179606749)
--(axis cs:0.717777820777429,13.557258540455)
--(axis cs:0.626099033699936,13.3739009663001)
--cycle;
\addplot [semithick, red, mark=*, mark size=1.5, mark options={solid}, only marks]
table {%
28 0
};
\addplot [semithick, red, mark=*, mark size=1.5, mark options={solid}, only marks]
table {%
0 28
};
\addplot [semithick, red, mark=*, mark size=1.5, mark options={solid}, only marks]
table {%
0 0
};
\addplot [semithick, green, mark=*, mark size=1.5, mark options={solid}, only marks]
table {%
4 5
};
\addplot [semithick, blue, mark=*, mark size=1.5, mark options={solid}, only marks]
table {%
8.00000000000001 9.99999999999999
};
\addplot [semithick, blue, mark=*, mark size=1.5, mark options={solid}, only marks]
table {%
16 5.99999999999997
};
\addplot [semithick, blue, mark=*, mark size=1.5, mark options={solid}, only marks]
table {%
0 14
};
\addplot [semithick, blue, mark=*, mark size=1.5, mark options={solid}, only marks]
table {%
4.00000000000003 11.9999999999999
};
\end{axis}

\end{tikzpicture}}}~\fbox{\adjustbox{width=0.31\textwidth}{
\begin{tikzpicture}

\definecolor{darkgray176}{RGB}{176,176,176}
\definecolor{green}{RGB}{0,128,0}

\begin{axis}[
hide x axis,
hide y axis,
tick align=outside,
tick pos=left,
title={\(\displaystyle \mathbf{\alpha}\)=(4,26,27)/57},
x grid style={darkgray176},
xmin=-11.8924242424242, xmax=73.5924242424242,
xtick style={color=black},
y grid style={darkgray176},
ymin=-1, ymax=62.7,
ytick style={color=black}
]
\path [draw=black, fill=black]
(axis cs:2,13)
--(axis cs:1.69664591741339,14.7767881980073)
--(axis cs:2.18588990775044,14.7015198918016)
--(axis cs:3.9258721226762,26.011404288819)
--(axis cs:4.0741278773238,25.988595711181)
--(axis cs:2.33414566239804,14.6787113141635)
--(axis cs:2.82338965273509,14.6034430079578)
--cycle;
\path [draw=black, fill=black]
(axis cs:6,39)
--(axis cs:6.30335408258661,37.2232118019927)
--(axis cs:5.81411009224956,37.2984801081984)
--(axis cs:4.0741278773238,25.988595711181)
--(axis cs:3.9258721226762,26.011404288819)
--(axis cs:5.66585433760196,37.3212886858365)
--(axis cs:5.17661034726491,37.3965569920422)
--cycle;
\path [draw=black, fill=black]
(axis cs:0,0)
--(axis cs:-0.303354082586612,1.7767881980073)
--(axis cs:0.185889907750443,1.7015198918016)
--(axis cs:1.9258721226762,13.011404288819)
--(axis cs:2.0741278773238,12.988595711181)
--(axis cs:0.334145662398036,1.67871131416351)
--(axis cs:0.823389652735091,1.60344300795781)
--cycle;
\path [draw=black, fill=black]
(axis cs:4,26)
--(axis cs:4.30335408258661,24.2232118019927)
--(axis cs:3.81411009224956,24.2984801081984)
--(axis cs:2.0741278773238,12.988595711181)
--(axis cs:1.9258721226762,13.011404288819)
--(axis cs:3.66585433760196,24.3212886858365)
--(axis cs:3.17661034726491,24.3965569920422)
--cycle;
\path [draw=black, fill=black]
(axis cs:57,0)
--(axis cs:55.3432426769179,0.710038852749477)
--(axis cs:55.6721027771387,1.08000646549789)
--(axis cs:29.9501727120878,23.9439443010987)
--(axis cs:30.0498272879122,24.0560556989013)
--(axis cs:55.7717573529632,1.19211786330044)
--(axis cs:56.100617453184,1.56208547604885)
--cycle;
\path [draw=black, fill=black]
(axis cs:3,48)
--(axis cs:4.65675732308211,47.2899611472505)
--(axis cs:4.3278972228613,46.9199935345021)
--(axis cs:30.0498272879122,24.0560556989013)
--(axis cs:29.9501727120878,23.9439443010987)
--(axis cs:4.22824264703681,46.8078821366996)
--(axis cs:3.899382546816,46.4379145239512)
--cycle;
\path [draw=black, fill=black]
(axis cs:0,0)
--(axis cs:0.97920948704156,1.51332375270059)
--(axis cs:1.288433535581,1.12679369202629)
--(axis cs:14.9531478714334,12.0585651607082)
--(axis cs:15.0468521285666,11.9414348392918)
--(axis cs:1.38213779271416,1.00966337060984)
--(axis cs:1.6913618412536,0.623133309935538)
--cycle;
\path [draw=black, fill=black]
(axis cs:30,24)
--(axis cs:29.0207905129584,22.4866762472994)
--(axis cs:28.711566464419,22.8732063079737)
--(axis cs:15.0468521285666,11.9414348392918)
--(axis cs:14.9531478714334,12.0585651607082)
--(axis cs:28.6178622072858,22.9903366293902)
--(axis cs:28.3086381587464,23.3768666900645)
--cycle;
\path [draw=black, fill=black]
(axis cs:15,12)
--(axis cs:13.9185010402224,13.4419986130368)
--(axis cs:14.3880992727574,13.5985313572151)
--(axis cs:8.92884875264621,29.9762829175487)
--(axis cs:9.07115124735379,30.0237170824513)
--(axis cs:14.530401767465,13.6459655221176)
--(axis cs:15,13.802498266296)
--cycle;
\path [draw=black, fill=black]
(axis cs:3,48)
--(axis cs:4.08149895977759,46.5580013869632)
--(axis cs:3.61190072724258,46.4014686427849)
--(axis cs:9.07115124735379,30.0237170824513)
--(axis cs:8.92884875264621,29.9762829175487)
--(axis cs:3.469598232535,46.3540344778824)
--(axis cs:3,46.197501733704)
--cycle;
\path [draw=black, fill=black]
(axis cs:9,30)
--(axis cs:7.91850104022241,31.4419986130368)
--(axis cs:8.38809927275742,31.5985313572151)
--(axis cs:5.92884875264621,38.9762829175487)
--(axis cs:6.07115124735379,39.0237170824513)
--(axis cs:8.530401767465,31.6459655221176)
--(axis cs:9,31.802498266296)
--cycle;
\path [draw=black, fill=black]
(axis cs:3,48)
--(axis cs:4.08149895977759,46.5580013869632)
--(axis cs:3.61190072724258,46.4014686427849)
--(axis cs:6.07115124735379,39.0237170824513)
--(axis cs:5.92884875264621,38.9762829175487)
--(axis cs:3.469598232535,46.3540344778824)
--(axis cs:3,46.197501733704)
--cycle;
\path [draw=black, fill=black]
(axis cs:0,57)
--(axis cs:1.08149895977759,55.5580013869632)
--(axis cs:0.611900727242581,55.4014686427849)
--(axis cs:3.07115124735379,48.0237170824513)
--(axis cs:2.92884875264621,47.9762829175487)
--(axis cs:0.469598232535004,55.3540344778824)
--(axis cs:4.85634478344628e-17,55.197501733704)
--cycle;
\path [draw=black, fill=black]
(axis cs:6,39)
--(axis cs:4.91850104022241,40.4419986130368)
--(axis cs:5.38809927275742,40.5985313572151)
--(axis cs:2.92884875264621,47.9762829175487)
--(axis cs:3.07115124735379,48.0237170824513)
--(axis cs:5.530401767465,40.6459655221176)
--(axis cs:6,40.802498266296)
--cycle;
\addplot [semithick, red, mark=*, mark size=1.5, mark options={solid}, only marks]
table {%
57 0
};
\addplot [semithick, red, mark=*, mark size=1.5, mark options={solid}, only marks]
table {%
0 57
};
\addplot [semithick, red, mark=*, mark size=1.5, mark options={solid}, only marks]
table {%
0 0
};
\addplot [semithick, green, mark=*, mark size=1.5, mark options={solid}, only marks]
table {%
4 26
};
\addplot [semithick, blue, mark=*, mark size=1.5, mark options={solid}, only marks]
table {%
2 13
};
\addplot [semithick, blue, mark=*, mark size=1.5, mark options={solid}, only marks]
table {%
30 24
};
\addplot [semithick, blue, mark=*, mark size=1.5, mark options={solid}, only marks]
table {%
15 12
};
\addplot [semithick, blue, mark=*, mark size=1.5, mark options={solid}, only marks]
table {%
9 30
};
\addplot [semithick, blue, mark=*, mark size=1.5, mark options={solid}, only marks]
table {%
6 39
};
\addplot [semithick, blue, mark=*, mark size=1.5, mark options={solid}, only marks]
table {%
3 48
};
\end{axis}

\end{tikzpicture}}}~\fbox{\adjustbox{width=0.31\textwidth}{
\begin{tikzpicture}

\definecolor{darkgray176}{RGB}{176,176,176}
\definecolor{green}{RGB}{0,128,0}

\begin{axis}[
hide x axis,
hide y axis,
tick align=outside,
tick pos=left,
title={\(\displaystyle \mathbf{\alpha}\)=(5,10,67)/82},
x grid style={darkgray176},
xmin=-16.5948051948052, xmax=105.794805194805,
xtick style={color=black},
y grid style={darkgray176},
ymin=-1, ymax=90.2,
ytick style={color=black}
]
\path [draw=black, fill=black]
(axis cs:3,6)
--(axis cs:3.36671514830997,8.56700603816976)
--(axis cs:4.0330634056049,8.23383190952229)
--(axis cs:4.93291796067501,10.0335410196625)
--(axis cs:5.06708203932499,9.9664589803375)
--(axis cs:4.16722748425489,8.1667498701973)
--(axis cs:4.83357574154983,7.83357574154983)
--cycle;
\path [draw=black, fill=black]
(axis cs:7,14)
--(axis cs:6.63328485169003,11.4329939618302)
--(axis cs:5.9669365943951,11.7661680904777)
--(axis cs:5.06708203932499,9.9664589803375)
--(axis cs:4.93291796067501,10.0335410196625)
--(axis cs:5.83277251574511,11.8332501298027)
--(axis cs:5.16642425845017,12.1664242584502)
--cycle;
\path [draw=black, fill=black]
(axis cs:0,0)
--(axis cs:0.366715148309965,2.56700603816976)
--(axis cs:1.0330634056049,2.23383190952229)
--(axis cs:2.93291796067501,6.0335410196625)
--(axis cs:3.06708203932499,5.9664589803375)
--(axis cs:1.16722748425489,2.1667498701973)
--(axis cs:1.83357574154983,1.83357574154983)
--cycle;
\path [draw=black, fill=black]
(axis cs:6,12)
--(axis cs:5.63328485169003,9.43299396183024)
--(axis cs:4.9669365943951,9.76616809047771)
--(axis cs:3.06708203932499,5.9664589803375)
--(axis cs:2.93291796067501,6.0335410196625)
--(axis cs:4.83277251574511,9.8332501298027)
--(axis cs:4.16642425845017,10.1664242584502)
--cycle;
\path [draw=black, fill=black]
(axis cs:82,0)
--(axis cs:79.4422137950426,-0.42629770082623)
--(axis cs:79.55840591228,0.309585708343927)
--(axis cs:43.9883028069895,5.92591777760032)
--(axis cs:44.0116971930105,6.07408222239968)
--(axis cs:79.581800298301,0.457750153143287)
--(axis cs:79.6979924155384,1.19363356231344)
--cycle;
\path [draw=black, fill=black]
(axis cs:6,12)
--(axis cs:8.55778620495738,12.4262977008262)
--(axis cs:8.44159408771999,11.6904142916561)
--(axis cs:44.0116971930105,6.07408222239968)
--(axis cs:43.9883028069895,5.92591777760032)
--(axis cs:8.41819970169903,11.5422498468567)
--(axis cs:8.30200758446164,10.8063664376866)
--cycle;
\path [draw=black, fill=black]
(axis cs:0,82)
--(axis cs:1.94219877962712,80.2819010795606)
--(axis cs:1.29745619694368,79.9086290580071)
--(axis cs:22.0649069714111,44.0375777202906)
--(axis cs:21.9350930285889,43.9624222797094)
--(axis cs:1.16764225412151,79.8334736174258)
--(axis cs:0.52289967143807,79.4602015958722)
--cycle;
\path [draw=black, fill=black]
(axis cs:44,6)
--(axis cs:42.0578012203729,7.71809892043937)
--(axis cs:42.7025438030563,8.09137094199294)
--(axis cs:21.9350930285889,43.9624222797094)
--(axis cs:22.0649069714111,44.0375777202906)
--(axis cs:42.8323577458785,8.1665263825742)
--(axis cs:43.4771003285619,8.53979840412777)
--cycle;
\path [draw=black, fill=black]
(axis cs:5,10)
--(axis cs:5.36671514830997,12.5670060381698)
--(axis cs:6.0330634056049,12.2338319095223)
--(axis cs:5.93291796067501,12.0335410196625)
--(axis cs:6.06708203932499,11.9664589803375)
--(axis cs:6.16722748425489,12.1667498701973)
--(axis cs:6.83357574154983,11.8335757415498)
--cycle;
\path [draw=black, fill=black]
(axis cs:7,14)
--(axis cs:6.63328485169003,11.4329939618302)
--(axis cs:5.9669365943951,11.7661680904777)
--(axis cs:6.06708203932499,11.9664589803375)
--(axis cs:5.93291796067501,12.0335410196625)
--(axis cs:5.83277251574511,11.8332501298027)
--(axis cs:5.16642425845017,12.1664242584502)
--cycle;
\path [draw=black, fill=black]
(axis cs:22,44)
--(axis cs:21.63328485169,41.4329939618302)
--(axis cs:20.9669365943951,41.7661680904777)
--(axis cs:14.067082039325,27.9664589803375)
--(axis cs:13.932917960675,28.0335410196625)
--(axis cs:20.8327725157451,41.8332501298027)
--(axis cs:20.1664242584502,42.1664242584502)
--cycle;
\path [draw=black, fill=black]
(axis cs:6,12)
--(axis cs:6.36671514830997,14.5670060381698)
--(axis cs:7.0330634056049,14.2338319095223)
--(axis cs:13.932917960675,28.0335410196625)
--(axis cs:14.067082039325,27.9664589803375)
--(axis cs:7.16722748425489,14.1667498701973)
--(axis cs:7.83357574154983,13.8335757415498)
--cycle;
\path [draw=black, fill=black]
(axis cs:0,0)
--(axis cs:0.366715148309965,2.56700603816976)
--(axis cs:1.0330634056049,2.23383190952229)
--(axis cs:6.93291796067501,14.0335410196625)
--(axis cs:7.06708203932499,13.9664589803375)
--(axis cs:1.16722748425489,2.1667498701973)
--(axis cs:1.83357574154983,1.83357574154983)
--cycle;
\path [draw=black, fill=black]
(axis cs:14,28)
--(axis cs:13.63328485169,25.4329939618302)
--(axis cs:12.9669365943951,25.7661680904777)
--(axis cs:7.06708203932499,13.9664589803375)
--(axis cs:6.93291796067501,14.0335410196625)
--(axis cs:12.8327725157451,25.8332501298027)
--(axis cs:12.1664242584502,26.1664242584502)
--cycle;
\addplot [semithick, red, mark=*, mark size=1.5, mark options={solid}, only marks]
table {%
82 0
};
\addplot [semithick, red, mark=*, mark size=1.5, mark options={solid}, only marks]
table {%
0 82
};
\addplot [semithick, red, mark=*, mark size=1.5, mark options={solid}, only marks]
table {%
0 0
};
\addplot [semithick, green, mark=*, mark size=1.5, mark options={solid}, only marks]
table {%
5 10
};
\addplot [semithick, blue, mark=*, mark size=1.5, mark options={solid}, only marks]
table {%
3 6
};
\addplot [semithick, blue, mark=*, mark size=1.5, mark options={solid}, only marks]
table {%
44 6
};
\addplot [semithick, blue, mark=*, mark size=1.5, mark options={solid}, only marks]
table {%
22 44
};
\addplot [semithick, blue, mark=*, mark size=1.5, mark options={solid}, only marks]
table {%
6 12
};
\addplot [semithick, blue, mark=*, mark size=1.5, mark options={solid}, only marks]
table {%
14 28
};
\addplot [semithick, blue, mark=*, mark size=1.5, mark options={solid}, only marks]
table {%
7 14
};
\end{axis}

\end{tikzpicture}}}\\
     \fbox{\adjustbox{width=0.31\textwidth}{
\begin{tikzpicture}

\definecolor{darkgray176}{RGB}{176,176,176}
\definecolor{green}{RGB}{0,128,0}

\begin{axis}[
hide x axis,
hide y axis,
tick align=outside,
tick pos=left,
title={\(\displaystyle \mathbf{\alpha}\)=(5,29,42)/76},
x grid style={darkgray176},
xmin=-15.4662337662338, xmax=98.0662337662338,
xtick style={color=black},
y grid style={darkgray176},
ymin=-1, ymax=83.6,
ytick style={color=black}
]
\path [draw=black, fill=black]
(axis cs:0,0)
--(axis cs:-0.361561918328271,2.37597832044292)
--(axis cs:0.313478260956607,2.25959208263519)
--(axis cs:4.92609049131917,29.0127430187381)
--(axis cs:5.07390950868083,28.9872569812619)
--(axis cs:0.46129727831826,2.23410604515904)
--(axis cs:1.13633745760314,2.1177198073513)
--cycle;
\path [draw=black, fill=black]
(axis cs:10,58)
--(axis cs:10.3615619183283,55.6240216795571)
--(axis cs:9.68652173904341,55.7404079173648)
--(axis cs:5.07390950868083,28.9872569812619)
--(axis cs:4.92609049131917,29.0127430187381)
--(axis cs:9.53870272168176,55.765893954841)
--(axis cs:8.86366254239688,55.8822801926487)
--cycle;
\path [draw=black, fill=black]
(axis cs:0,76)
--(axis cs:1.77162541293965,74.3760100381387)
--(axis cs:1.1728276787964,74.0433446302813)
--(axis cs:10.0655617957091,58.0364232198384)
--(axis cs:9.9344382042909,57.9635767801616)
--(axis cs:1.04170408737817,73.9704981906045)
--(axis cs:0.442906353234916,73.6378327827471)
--cycle;
\path [draw=black, fill=black]
(axis cs:20,40.0000000000001)
--(axis cs:18.2283745870604,41.6239899618614)
--(axis cs:18.8271723212036,41.9566553697188)
--(axis cs:9.9344382042909,57.9635767801616)
--(axis cs:10.0655617957091,58.0364232198384)
--(axis cs:18.9582959126219,42.0295018093956)
--(axis cs:19.5570936467651,42.3621672172529)
--cycle;
\path [draw=black, fill=black]
(axis cs:48,20)
--(axis cs:45.6500171251273,19.4964322410987)
--(axis cs:45.7256627204886,20.1772425993502)
--(axis cs:29.9917176355444,21.9254587198996)
--(axis cs:30.0082823644557,22.0745412801006)
--(axis cs:45.7422274493998,20.3263251595512)
--(axis cs:45.8178730447611,21.0071355178026)
--cycle;
\path [draw=black, fill=black]
(axis cs:12.0000000000001,24.0000000000002)
--(axis cs:14.3499828748729,24.5035677589015)
--(axis cs:14.2743372795116,23.8227574006501)
--(axis cs:30.0082823644557,22.0745412801006)
--(axis cs:29.9917176355444,21.9254587198996)
--(axis cs:14.2577725506004,23.673674840449)
--(axis cs:14.1821269552391,22.9928644821976)
--cycle;
\path [draw=black, fill=black]
(axis cs:10,58)
--(axis cs:11.7716254129397,56.3760100381387)
--(axis cs:11.1728276787964,56.0433446302813)
--(axis cs:20.0655617957091,40.0364232198385)
--(axis cs:19.9344382042909,39.9635767801617)
--(axis cs:11.0417040873782,55.9704981906045)
--(axis cs:10.4429063532349,55.6378327827472)
--cycle;
\path [draw=black, fill=black]
(axis cs:30.0000000000001,22.0000000000001)
--(axis cs:28.2283745870604,23.6239899618615)
--(axis cs:28.8271723212037,23.9566553697188)
--(axis cs:19.9344382042909,39.9635767801617)
--(axis cs:20.0655617957091,40.0364232198385)
--(axis cs:28.9582959126219,24.0295018093956)
--(axis cs:29.5570936467651,24.362167217253)
--cycle;
\path [draw=black, fill=black]
(axis cs:76,0)
--(axis cs:73.7029466584221,0.706785643562423)
--(axis cs:74.1010948211197,1.26419307133904)
--(axis cs:47.9564071354711,19.9389699896595)
--(axis cs:48.0435928645289,20.0610300103405)
--(axis cs:74.1882805501776,1.38625309202005)
--(axis cs:74.5864287128752,1.94366051979666)
--cycle;
\path [draw=black, fill=black]
(axis cs:20,40.0000000000001)
--(axis cs:22.2970533415779,39.2932143564376)
--(axis cs:21.8989051788803,38.735806928661)
--(axis cs:48.0435928645289,20.0610300103405)
--(axis cs:47.9564071354711,19.9389699896595)
--(axis cs:21.8117194498225,38.61374690798)
--(axis cs:21.4135712871249,38.0563394802034)
--cycle;
\path [draw=black, fill=black]
(axis cs:0,76)
--(axis cs:2.06083934927725,74.7634963904337)
--(axis cs:1.54074857665703,74.3177042996163)
--(axis cs:24.0569442451776,48.0488093530097)
--(axis cs:23.9430557548228,47.9511906469913)
--(axis cs:1.42686008630223,74.2200855935979)
--(axis cs:0.906769313682011,73.7742935027806)
--cycle;
\path [draw=black, fill=black]
(axis cs:48,20)
--(axis cs:45.9391606507228,21.2365036095664)
--(axis cs:46.459251423343,21.6822957003837)
--(axis cs:23.9430557548228,47.9511906469913)
--(axis cs:24.0569442451776,48.0488093530097)
--(axis cs:46.5731399136978,21.7799144064021)
--(axis cs:47.0932306863181,22.2257064972195)
--cycle;
\path [draw=black, fill=black]
(axis cs:0,0)
--(axis cs:0.339882332579968,2.37917632805978)
--(axis cs:0.95256495841491,2.07283501514231)
--(axis cs:11.9329179606751,24.0335410196627)
--(axis cs:12.0670820393251,23.9664589803377)
--(axis cs:1.0867290370649,2.00575297581731)
--(axis cs:1.69941166289984,1.69941166289984)
--cycle;
\path [draw=black, fill=black]
(axis cs:24.0000000000002,48.0000000000005)
--(axis cs:23.6601176674203,45.6208236719407)
--(axis cs:23.0474350415853,45.9271649848582)
--(axis cs:12.0670820393251,23.9664589803377)
--(axis cs:11.9329179606751,24.0335410196627)
--(axis cs:22.9132709629353,45.9942470241831)
--(axis cs:22.3005883371004,46.3005883371006)
--cycle;
\addplot [semithick, red, mark=*, mark size=1.5, mark options={solid}, only marks]
table {%
76 0
};
\addplot [semithick, red, mark=*, mark size=1.5, mark options={solid}, only marks]
table {%
0 76
};
\addplot [semithick, red, mark=*, mark size=1.5, mark options={solid}, only marks]
table {%
0 0
};
\addplot [semithick, green, mark=*, mark size=1.5, mark options={solid}, only marks]
table {%
5 29
};
\addplot [semithick, blue, mark=*, mark size=1.5, mark options={solid}, only marks]
table {%
10 58
};
\addplot [semithick, blue, mark=*, mark size=1.5, mark options={solid}, only marks]
table {%
30.0000000000001 22.0000000000001
};
\addplot [semithick, blue, mark=*, mark size=1.5, mark options={solid}, only marks]
table {%
20 40.0000000000001
};
\addplot [semithick, blue, mark=*, mark size=1.5, mark options={solid}, only marks]
table {%
48 20
};
\addplot [semithick, blue, mark=*, mark size=1.5, mark options={solid}, only marks]
table {%
24.0000000000002 48.0000000000005
};
\addplot [semithick, blue, mark=*, mark size=1.5, mark options={solid}, only marks]
table {%
12.0000000000001 24.0000000000002
};
\end{axis}

\end{tikzpicture}}}~\fbox{\adjustbox{width=0.31\textwidth}{
\begin{tikzpicture}

\definecolor{darkgray176}{RGB}{176,176,176}
\definecolor{green}{RGB}{0,128,0}

\begin{axis}[
hide x axis,
hide y axis,
tick align=outside,
tick pos=left,
title={\(\displaystyle \mathbf{\alpha}\)=(5,32,36)/73},
x grid style={darkgray176},
xmin=-14.9019480519481, xmax=94.2019480519481,
xtick style={color=black},
y grid style={darkgray176},
ymin=-1, ymax=80.3,
ytick style={color=black}
]
\path [draw=black, fill=black]
(axis cs:8,22)
--(axis cs:6.67149494239826,23.8878756081709)
--(axis cs:7.29887115921812,24.0760884732168)
--(axis cs:4.92816302860841,31.9784489085825)
--(axis cs:5.07183697139159,32.0215510914175)
--(axis cs:7.44254510200129,24.1191906560518)
--(axis cs:8.06992131882114,24.3074035210978)
--cycle;
\path [draw=black, fill=black]
(axis cs:2,42)
--(axis cs:3.32850505760174,40.1121243918291)
--(axis cs:2.70112884078188,39.9239115267832)
--(axis cs:5.07183697139159,32.0215510914175)
--(axis cs:4.92816302860841,31.9784489085825)
--(axis cs:2.55745489799871,39.8808093439482)
--(axis cs:1.93007868117886,39.6925964789022)
--cycle;
\path [draw=black, fill=black]
(axis cs:12,33)
--(axis cs:11.9376318909867,30.692379966507)
--(axis cs:11.3220671986018,30.9162216728288)
--(axis cs:8.07048450676163,21.9743692702685)
--(axis cs:7.92951549323837,22.0256307297315)
--(axis cs:11.1810981850785,30.9674831322918)
--(axis cs:10.5655334926936,31.1913248386136)
--cycle;
\path [draw=black, fill=black]
(axis cs:4,11)
--(axis cs:4.06236810901332,13.307620033493)
--(axis cs:4.67793280139825,13.0837783271712)
--(axis cs:7.92951549323837,22.0256307297315)
--(axis cs:8.07048450676163,21.9743692702685)
--(axis cs:4.81890181492152,13.0325168677082)
--(axis cs:5.43446650730644,12.8086751613864)
--cycle;
\path [draw=black, fill=black]
(axis cs:73,0)
--(axis cs:70.707615205659,0.271977856955711)
--(axis cs:70.9865122644551,0.864634106897362)
--(axis cs:38.9680652222753,15.9321385973349)
--(axis cs:39.0319347777247,16.0678614026651)
--(axis cs:71.0503818199046,1.00035691222751)
--(axis cs:71.3292788787006,1.59301316216916)
--cycle;
\path [draw=black, fill=black]
(axis cs:5,32)
--(axis cs:7.29238479434099,31.7280221430443)
--(axis cs:7.01348773554492,31.1353658931026)
--(axis cs:39.0319347777247,16.0678614026651)
--(axis cs:38.9680652222753,15.9321385973349)
--(axis cs:6.94961818009544,30.9996430877725)
--(axis cs:6.67072112129937,30.4069868378308)
--cycle;
\path [draw=black, fill=black]
(axis cs:5,32)
--(axis cs:7.29238479434099,31.7280221430443)
--(axis cs:7.01348773554492,31.1353658931026)
--(axis cs:22.0319347777247,24.0678614026651)
--(axis cs:21.9680652222753,23.9321385973349)
--(axis cs:6.94961818009544,30.9996430877725)
--(axis cs:6.67072112129937,30.4069868378308)
--cycle;
\path [draw=black, fill=black]
(axis cs:39,16)
--(axis cs:36.707615205659,16.2719778569557)
--(axis cs:36.9865122644551,16.8646341068974)
--(axis cs:21.9680652222753,23.9321385973349)
--(axis cs:22.0319347777247,24.0678614026651)
--(axis cs:37.0503818199046,17.0003569122275)
--(axis cs:37.3292788787006,17.5930131621692)
--cycle;
\path [draw=black, fill=black]
(axis cs:22,24)
--(axis cs:19.8838415656343,24.9224280354927)
--(axis cs:20.322013464847,25.4092857012846)
--(axis cs:11.9498276451283,32.9442529390315)
--(axis cs:12.0501723548717,33.0557470609685)
--(axis cs:20.4223581745904,25.5207798232217)
--(axis cs:20.8605300738031,26.0076374890136)
--cycle;
\path [draw=black, fill=black]
(axis cs:2,42)
--(axis cs:4.11615843436568,41.0775719645073)
--(axis cs:3.67798653515298,40.5907142987154)
--(axis cs:12.0501723548717,33.0557470609685)
--(axis cs:11.9498276451283,32.9442529390315)
--(axis cs:3.57764182540961,40.4792201767783)
--(axis cs:3.13946992619691,39.9923625109864)
--cycle;
\path [draw=black, fill=black]
(axis cs:0,0)
--(axis cs:0.0623681090133236,2.30762003349297)
--(axis cs:0.67793280139825,2.08377832717118)
--(axis cs:3.92951549323837,11.0256307297315)
--(axis cs:4.07048450676163,10.9743692702685)
--(axis cs:0.818901814921515,2.03251686770817)
--(axis cs:1.43446650730644,1.80867516138638)
--cycle;
\path [draw=black, fill=black]
(axis cs:8,22)
--(axis cs:7.93763189098668,19.692379966507)
--(axis cs:7.32206719860175,19.9162216728288)
--(axis cs:4.07048450676163,10.9743692702685)
--(axis cs:3.92951549323837,11.0256307297315)
--(axis cs:7.18109818507848,19.9674831322918)
--(axis cs:6.56553349269356,20.1913248386136)
--cycle;
\path [draw=black, fill=black]
(axis cs:0,73)
--(axis cs:0.869482664362518,70.8615426362976)
--(axis cs:0.215841586988178,70.8193722442089)
--(axis cs:2.07484439817263,42.0048286708498)
--(axis cs:1.92515560182737,41.9951713291502)
--(axis cs:0.0661527906429091,70.8097149025092)
--(axis cs:-0.587488286731431,70.7675445104206)
--cycle;
\path [draw=black, fill=black]
(axis cs:4,11)
--(axis cs:3.13051733563748,13.1384573637024)
--(axis cs:3.78415841301182,13.1806277557911)
--(axis cs:1.92515560182737,41.9951713291502)
--(axis cs:2.07484439817263,42.0048286708498)
--(axis cs:3.93384720935709,13.1902850974908)
--(axis cs:4.58748828673143,13.2324554895794)
--cycle;
\addplot [semithick, red, mark=*, mark size=1.5, mark options={solid}, only marks]
table {%
73 0
};
\addplot [semithick, red, mark=*, mark size=1.5, mark options={solid}, only marks]
table {%
0 73
};
\addplot [semithick, red, mark=*, mark size=1.5, mark options={solid}, only marks]
table {%
0 0
};
\addplot [semithick, green, mark=*, mark size=1.5, mark options={solid}, only marks]
table {%
5 32
};
\addplot [semithick, blue, mark=*, mark size=1.5, mark options={solid}, only marks]
table {%
8 22
};
\addplot [semithick, blue, mark=*, mark size=1.5, mark options={solid}, only marks]
table {%
39 16
};
\addplot [semithick, blue, mark=*, mark size=1.5, mark options={solid}, only marks]
table {%
22 24
};
\addplot [semithick, blue, mark=*, mark size=1.5, mark options={solid}, only marks]
table {%
12 33
};
\addplot [semithick, blue, mark=*, mark size=1.5, mark options={solid}, only marks]
table {%
4 11
};
\addplot [semithick, blue, mark=*, mark size=1.5, mark options={solid}, only marks]
table {%
2 42
};
\end{axis}

\end{tikzpicture}}}~\fbox{\adjustbox{width=0.31\textwidth}{
\begin{tikzpicture}

\definecolor{darkgray176}{RGB}{176,176,176}
\definecolor{green}{RGB}{0,128,0}

\begin{axis}[
hide x axis,
hide y axis,
tick align=outside,
tick pos=left,
title={\(\displaystyle \mathbf{\alpha}\)=(13,24,39)/76},
x grid style={darkgray176},
xmin=-15.4662337662338, xmax=98.0662337662338,
xtick style={color=black},
y grid style={darkgray176},
ymin=-1, ymax=83.6,
ytick style={color=black}
]
\path [draw=black, fill=black]
(axis cs:0,0)
--(axis cs:0.417663684104968,2.36676087659482)
--(axis cs:1.01997868118266,2.0405069198444)
--(axis cs:12.9340531025097,24.0357212361406)
--(axis cs:13.0659468974903,23.9642787638594)
--(axis cs:1.15187247616317,1.96906444756329)
--(axis cs:1.75418747324086,1.64281049081287)
--cycle;
\path [draw=black, fill=black]
(axis cs:26,48)
--(axis cs:25.582336315895,45.6332391234052)
--(axis cs:24.9800213188173,45.9594930801556)
--(axis cs:13.0659468974903,23.9642787638594)
--(axis cs:12.9340531025097,24.0357212361406)
--(axis cs:24.8481275238368,46.0309355524367)
--(axis cs:24.2458125267591,46.3571895091871)
--cycle;
\path [draw=black, fill=black]
(axis cs:0,76)
--(axis cs:2.10835156019338,74.8463736746112)
--(axis cs:1.60638801842937,74.3802646715446)
--(axis cs:26.054959511872,48.0510338324525)
--(axis cs:25.945040488128,47.9489661675475)
--(axis cs:1.49646899468542,74.2781970066395)
--(axis cs:0.994505452921406,73.8120880035729)
--cycle;
\path [draw=black, fill=black]
(axis cs:52,20)
--(axis cs:49.8916484398066,21.1536263253888)
--(axis cs:50.3936119815706,21.6197353284554)
--(axis cs:25.945040488128,47.9489661675475)
--(axis cs:26.054959511872,48.0510338324525)
--(axis cs:50.5035310053146,21.7218029933605)
--(axis cs:51.0054945470786,22.1879119964271)
--cycle;
\path [draw=black, fill=black]
(axis cs:76,0)
--(axis cs:73.761915338773,0.875772258741009)
--(axis cs:74.2004416525431,1.40200383526521)
--(axis cs:51.9519861700252,19.9423834040302)
--(axis cs:52.0480138299748,20.0576165959698)
--(axis cs:74.2964693124928,1.51723702720482)
--(axis cs:74.734995626263,2.04346860372902)
--cycle;
\path [draw=black, fill=black]
(axis cs:28,40)
--(axis cs:30.238084661227,39.124227741259)
--(axis cs:29.7995583474569,38.5979961647348)
--(axis cs:52.0480138299748,20.0576165959698)
--(axis cs:51.9519861700252,19.9423834040302)
--(axis cs:29.7035306875072,38.4827629727952)
--(axis cs:29.265004373737,37.956531396271)
--cycle;
\path [draw=black, fill=black]
(axis cs:38,0)
--(axis cs:36.7097104748067,2.02759782530374)
--(axis cs:37.3742580874063,2.19373472845363)
--(axis cs:33.9272393124891,15.9818098281223)
--(axis cs:34.0727606875109,16.0181901718777)
--(axis cs:37.5197794624281,2.23011507220908)
--(axis cs:38.1843270750276,2.39625197535897)
--cycle;
\path [draw=black, fill=black]
(axis cs:30,32)
--(axis cs:31.2902895251933,29.9724021746963)
--(axis cs:30.6257419125937,29.8062652715464)
--(axis cs:34.0727606875109,16.0181901718777)
--(axis cs:33.9272393124891,15.9818098281223)
--(axis cs:30.4802205375719,29.7698849277909)
--(axis cs:29.8156729249724,29.603748024641)
--cycle;
\path [draw=black, fill=black]
(axis cs:26,48)
--(axis cs:27.2902895251933,45.9724021746963)
--(axis cs:26.6257419125937,45.8062652715464)
--(axis cs:28.0727606875109,40.0181901718777)
--(axis cs:27.9272393124891,39.9818098281223)
--(axis cs:26.4802205375719,45.7698849277909)
--(axis cs:25.8156729249724,45.603748024641)
--cycle;
\path [draw=black, fill=black]
(axis cs:30,32)
--(axis cs:28.7097104748067,34.0275978253037)
--(axis cs:29.3742580874063,34.1937347284536)
--(axis cs:27.9272393124891,39.9818098281223)
--(axis cs:28.0727606875109,40.0181901718777)
--(axis cs:29.5197794624281,34.2301150722091)
--(axis cs:30.1843270750276,34.396251975359)
--cycle;
\path [draw=black, fill=black]
(axis cs:76,0)
--(axis cs:73.72,-0.76)
--(axis cs:73.72,-0.075)
--(axis cs:38,-0.075)
--(axis cs:38,0.075)
--(axis cs:73.72,0.075)
--(axis cs:73.72,0.76)
--cycle;
\path [draw=black, fill=black]
(axis cs:0,0)
--(axis cs:2.28,0.76)
--(axis cs:2.28,0.075)
--(axis cs:38,0.075)
--(axis cs:38,-0.075)
--(axis cs:2.28,-0.075)
--(axis cs:2.28,-0.76)
--cycle;
\path [draw=black, fill=black]
(axis cs:26,48)
--(axis cs:27.2902895251933,45.9724021746963)
--(axis cs:26.6257419125937,45.8062652715464)
--(axis cs:30.0727606875109,32.0181901718777)
--(axis cs:29.9272393124891,31.9818098281223)
--(axis cs:26.4802205375719,45.7698849277909)
--(axis cs:25.8156729249724,45.603748024641)
--cycle;
\path [draw=black, fill=black]
(axis cs:34,16)
--(axis cs:32.7097104748067,18.0275978253037)
--(axis cs:33.3742580874063,18.1937347284536)
--(axis cs:29.9272393124891,31.9818098281223)
--(axis cs:30.0727606875109,32.0181901718777)
--(axis cs:33.5197794624281,18.2301150722091)
--(axis cs:34.1843270750276,18.396251975359)
--cycle;
\addplot [semithick, red, mark=*, mark size=1.5, mark options={solid}, only marks]
table {%
76 0
};
\addplot [semithick, red, mark=*, mark size=1.5, mark options={solid}, only marks]
table {%
0 76
};
\addplot [semithick, red, mark=*, mark size=1.5, mark options={solid}, only marks]
table {%
0 0
};
\addplot [semithick, green, mark=*, mark size=1.5, mark options={solid}, only marks]
table {%
13 24
};
\addplot [semithick, blue, mark=*, mark size=1.5, mark options={solid}, only marks]
table {%
26 48
};
\addplot [semithick, blue, mark=*, mark size=1.5, mark options={solid}, only marks]
table {%
52 20
};
\addplot [semithick, blue, mark=*, mark size=1.5, mark options={solid}, only marks]
table {%
34 16
};
\addplot [semithick, blue, mark=*, mark size=1.5, mark options={solid}, only marks]
table {%
28 40
};
\addplot [semithick, blue, mark=*, mark size=1.5, mark options={solid}, only marks]
table {%
38 0
};
\addplot [semithick, blue, mark=*, mark size=1.5, mark options={solid}, only marks]
table {%
30 32
};
\end{axis}

\end{tikzpicture}}}\\
     \fbox{\adjustbox{width=0.31\textwidth}{

\begin{tikzpicture}

\definecolor{darkgray176}{RGB}{176,176,176}
\definecolor{green}{RGB}{0,128,0}

\begin{axis}[
hide x axis,
hide y axis,
tick align=outside,
tick pos=left,
title={$\displaystyle \mathbf{\alpha}=(10,13,26)/49$},
x grid style={darkgray176},
xmin=-10.3876623376623, xmax=63.2876623376623,
xtick style={color=black},
y grid style={darkgray176},
ymin=-1, ymax=53.9,
ytick style={color=black}
]
\path [draw=black, fill=black]
(axis cs:16,11)
--(axis cs:14.4504839465175,11)
--(axis cs:14.5817184694145,11.393703568691)
--(axis cs:9.97628291754874,12.9288487526462)
--(axis cs:10.0237170824513,13.0711512473538)
--(axis cs:14.629152634317,11.5360060633985)
--(axis cs:14.760387157214,11.9297096320895)
--cycle;
\path [draw=black, fill=black]
(axis cs:4,15)
--(axis cs:5.54951605348251,15)
--(axis cs:5.41828153058552,14.606296431309)
--(axis cs:10.0237170824513,13.0711512473538)
--(axis cs:9.97628291754874,12.9288487526462)
--(axis cs:5.37084736568299,14.4639939366015)
--(axis cs:5.239612842786,14.0702903679105)
--cycle;
\path [draw=black, fill=black]
(axis cs:28,7)
--(axis cs:26.4504839465175,7)
--(axis cs:26.5817184694145,7.39370356869096)
--(axis cs:15.9762829175487,10.9288487526462)
--(axis cs:16.0237170824513,11.0711512473538)
--(axis cs:26.629152634317,7.53600606339854)
--(axis cs:26.760387157214,7.9297096320895)
--cycle;
\path [draw=black, fill=black]
(axis cs:4,15)
--(axis cs:5.54951605348251,15)
--(axis cs:5.41828153058552,14.606296431309)
--(axis cs:16.0237170824513,11.0711512473538)
--(axis cs:15.9762829175487,10.9288487526462)
--(axis cs:5.37084736568299,14.4639939366015)
--(axis cs:5.239612842786,14.0702903679105)
--cycle;
\path [draw=black, fill=black]
(axis cs:49,0)
--(axis cs:47.4504839465175,-3.04318055372964e-17)
--(axis cs:47.5817184694145,0.393703568690963)
--(axis cs:27.9762829175487,6.92884875264621)
--(axis cs:28.0237170824513,7.07115124735379)
--(axis cs:47.629152634317,0.53600606339854)
--(axis cs:47.760387157214,0.929709632089504)
--cycle;
\path [draw=black, fill=black]
(axis cs:7,14)
--(axis cs:8.54951605348251,14)
--(axis cs:8.41828153058552,13.606296431309)
--(axis cs:28.0237170824513,7.07115124735379)
--(axis cs:27.9762829175487,6.92884875264621)
--(axis cs:8.37084736568299,13.4639939366015)
--(axis cs:8.23961284278601,13.0702903679105)
--cycle;
\path [draw=black, fill=black]
(axis cs:0,0)
--(axis cs:-0.0946910240809222,1.54662005998839)
--(axis cs:0.306296475853595,1.43969006000586)
--(axis cs:3.92753237952991,15.019324698792)
--(axis cs:4.07246762047009,14.980675301208)
--(axis cs:0.451231716793782,1.40104066242181)
--(axis cs:0.852219216728299,1.29411066243927)
--cycle;
\path [draw=black, fill=black]
(axis cs:8,30)
--(axis cs:8.09469102408092,28.4533799400116)
--(axis cs:7.6937035241464,28.5603099399941)
--(axis cs:4.07246762047009,14.980675301208)
--(axis cs:3.92753237952991,15.019324698792)
--(axis cs:7.54876828320622,28.5989593375782)
--(axis cs:7.1477807832717,28.7058893375607)
--cycle;
\path [draw=black, fill=black]
(axis cs:0,49)
--(axis cs:1.02204512390311,47.8353439285755)
--(axis cs:0.63956644322081,47.6743002735514)
--(axis cs:8.06912265313536,30.0291042750044)
--(axis cs:7.93087734686465,29.9708957249956)
--(axis cs:0.5013211369501,47.6160917235427)
--(axis cs:0.118842456267803,47.4550480685186)
--cycle;
\path [draw=black, fill=black]
(axis cs:16,11)
--(axis cs:14.9779548760969,12.1646560714245)
--(axis cs:15.3604335567792,12.3256997264486)
--(axis cs:7.93087734686464,29.9708957249956)
--(axis cs:8.06912265313536,30.0291042750044)
--(axis cs:15.4986788630499,12.3839082764573)
--(axis cs:15.8811575437322,12.5449519314814)
--cycle;
\path [draw=black, fill=black]
(axis cs:10,13)
--(axis cs:8.45048394651749,13)
--(axis cs:8.58171846941448,13.393703568691)
--(axis cs:6.97628291754874,13.9288487526462)
--(axis cs:7.02371708245126,14.0711512473538)
--(axis cs:8.62915263431701,13.5360060633985)
--(axis cs:8.760387157214,13.9297096320895)
--cycle;
\path [draw=black, fill=black]
(axis cs:4,15)
--(axis cs:5.54951605348251,15)
--(axis cs:5.41828153058552,14.606296431309)
--(axis cs:7.02371708245126,14.0711512473538)
--(axis cs:6.97628291754874,13.9288487526462)
--(axis cs:5.37084736568299,14.4639939366015)
--(axis cs:5.239612842786,14.0702903679105)
--cycle;
\addplot [semithick, red, mark=*, mark size=1.5, mark options={solid}, only marks]
table {%
49 0
};
\addplot [semithick, red, mark=*, mark size=1.5, mark options={solid}, only marks]
table {%
0 49
};
\addplot [semithick, red, mark=*, mark size=1.5, mark options={solid}, only marks]
table {%
0 0
};
\addplot [semithick, green, mark=*, mark size=1.5, mark options={solid}, only marks]
table {%
10 13
};
\addplot [semithick, blue, mark=*, mark size=1.5, mark options={solid}, only marks]
table {%
16 11
};
\addplot [semithick, blue, mark=*, mark size=1.5, mark options={solid}, only marks]
table {%
28 7
};
\addplot [semithick, blue, mark=*, mark size=1.5, mark options={solid}, only marks]
table {%
4 15
};
\addplot [semithick, blue, mark=*, mark size=1.5, mark options={solid}, only marks]
table {%
8 30
};
\addplot [semithick, blue, mark=*, mark size=1.5, mark options={solid}, only marks]
table {%
7 14
};
\end{axis}

\end{tikzpicture}
}}~\fbox{\adjustbox{width=0.31\textwidth}{
\begin{tikzpicture}

\definecolor{darkgray176}{RGB}{176,176,176}
\definecolor{green}{RGB}{0,128,0}

\begin{axis}[
hide x axis,
hide y axis,
tick align=outside,
tick pos=left,
title={\(\displaystyle \mathbf{\alpha}\)=(12,29,35)/76},
x grid style={darkgray176},
xmin=-15.4662337662338, xmax=98.0662337662338,
xtick style={color=black},
y grid style={darkgray176},
ymin=-1, ymax=83.6,
ytick style={color=black}
]
\path [draw=black, fill=black]
(axis cs:0,38)
--(axis cs:2.28,37.24)
--(axis cs:1.869,36.692)
--(axis cs:12.045,29.06)
--(axis cs:11.955,28.94)
--(axis cs:1.779,36.572)
--(axis cs:1.368,36.024)
--cycle;
\path [draw=black, fill=black]
(axis cs:24,20)
--(axis cs:21.72,20.76)
--(axis cs:22.131,21.308)
--(axis cs:11.955,28.94)
--(axis cs:12.045,29.06)
--(axis cs:22.221,21.428)
--(axis cs:22.632,21.976)
--cycle;
\path [draw=black, fill=black]
(axis cs:0,76)
--(axis cs:0.76,73.72)
--(axis cs:0.075,73.72)
--(axis cs:0.075,38)
--(axis cs:-0.075,38)
--(axis cs:-0.075,73.72)
--(axis cs:-0.76,73.72)
--cycle;
\path [draw=black, fill=black]
(axis cs:0,0)
--(axis cs:-0.76,2.28)
--(axis cs:-0.075,2.28)
--(axis cs:-0.075,38)
--(axis cs:0.075,38)
--(axis cs:0.075,2.28)
--(axis cs:0.76,2.28)
--cycle;
\path [draw=black, fill=black]
(axis cs:76,0)
--(axis cs:73.6300061340395,-0.398909858627007)
--(axis cs:73.735754297027,0.277878384492497)
--(axis cs:43.9884217339795,4.92589909746867)
--(axis cs:44.0115782660205,5.07410090253133)
--(axis cs:73.758910829068,0.426080189555163)
--(axis cs:73.8646589920554,1.10286843267467)
--cycle;
\path [draw=black, fill=black]
(axis cs:12,10)
--(axis cs:14.3699938659605,10.398909858627)
--(axis cs:14.264245702973,9.7221216155075)
--(axis cs:44.0115782660205,5.07410090253133)
--(axis cs:43.9884217339795,4.92589909746867)
--(axis cs:14.241089170932,9.57391981044484)
--(axis cs:14.1353410079446,8.89713156732533)
--cycle;
\path [draw=black, fill=black]
(axis cs:44,5)
--(axis cs:41.72,5.76000000000001)
--(axis cs:42.131,6.308)
--(axis cs:35.955,10.94)
--(axis cs:36.045,11.06)
--(axis cs:42.221,6.428)
--(axis cs:42.632,6.976)
--cycle;
\path [draw=black, fill=black]
(axis cs:28,17.0000000000001)
--(axis cs:30.28,16.2400000000001)
--(axis cs:29.869,15.6920000000001)
--(axis cs:36.045,11.06)
--(axis cs:35.955,10.94)
--(axis cs:29.779,15.5720000000001)
--(axis cs:29.368,15.0240000000001)
--cycle;
\path [draw=black, fill=black]
(axis cs:12,29)
--(axis cs:14.28,28.24)
--(axis cs:13.869,27.692)
--(axis cs:28.045,17.0600000000001)
--(axis cs:27.955,16.9400000000001)
--(axis cs:13.779,27.572)
--(axis cs:13.368,27.024)
--cycle;
\path [draw=black, fill=black]
(axis cs:44,5)
--(axis cs:41.72,5.76000000000001)
--(axis cs:42.131,6.308)
--(axis cs:27.955,16.9400000000001)
--(axis cs:28.045,17.0600000000001)
--(axis cs:42.221,6.428)
--(axis cs:42.632,6.976)
--cycle;
\path [draw=black, fill=black]
(axis cs:0,0)
--(axis cs:1.26500437373701,2.04346860372902)
--(axis cs:1.70353068750718,1.51723702720482)
--(axis cs:11.9519861700252,10.0576165959698)
--(axis cs:12.0480138299748,9.9423834040302)
--(axis cs:1.79955834745685,1.40200383526521)
--(axis cs:2.23808466122702,0.875772258741009)
--cycle;
\path [draw=black, fill=black]
(axis cs:24,20)
--(axis cs:22.734995626263,17.956531396271)
--(axis cs:22.2964693124928,18.4827629727952)
--(axis cs:12.0480138299748,9.9423834040302)
--(axis cs:11.9519861700252,10.0576165959698)
--(axis cs:22.2004416525431,18.5979961647348)
--(axis cs:21.761915338773,19.124227741259)
--cycle;
\path [draw=black, fill=black]
(axis cs:12,29)
--(axis cs:14.28,28.24)
--(axis cs:13.869,27.692)
--(axis cs:24.045,20.06)
--(axis cs:23.955,19.94)
--(axis cs:13.779,27.572)
--(axis cs:13.368,27.024)
--cycle;
\path [draw=black, fill=black]
(axis cs:36,11)
--(axis cs:33.72,11.76)
--(axis cs:34.131,12.308)
--(axis cs:23.955,19.94)
--(axis cs:24.045,20.06)
--(axis cs:34.221,12.428)
--(axis cs:34.632,12.976)
--cycle;
\addplot [semithick, red, mark=*, mark size=1.5, mark options={solid}, only marks]
table {%
76 0
};
\addplot [semithick, red, mark=*, mark size=1.5, mark options={solid}, only marks]
table {%
0 76
};
\addplot [semithick, red, mark=*, mark size=1.5, mark options={solid}, only marks]
table {%
0 0
};
\addplot [semithick, green, mark=*, mark size=1.5, mark options={solid}, only marks]
table {%
12 29
};
\addplot [semithick, blue, mark=*, mark size=1.5, mark options={solid}, only marks]
table {%
0 38
};
\addplot [semithick, blue, mark=*, mark size=1.5, mark options={solid}, only marks]
table {%
44 5
};
\addplot [semithick, blue, mark=*, mark size=1.5, mark options={solid}, only marks]
table {%
36 11
};
\addplot [semithick, blue, mark=*, mark size=1.5, mark options={solid}, only marks]
table {%
28 17.0000000000001
};
\addplot [semithick, blue, mark=*, mark size=1.5, mark options={solid}, only marks]
table {%
12 10
};
\addplot [semithick, blue, mark=*, mark size=1.5, mark options={solid}, only marks]
table {%
24 20
};
\end{axis}

\end{tikzpicture}}}~\fbox{\adjustbox{width=0.31\textwidth}{
\begin{tikzpicture}

\definecolor{darkgray176}{RGB}{176,176,176}
\definecolor{green}{RGB}{0,128,0}

\begin{axis}[
hide x axis,
hide y axis,
tick align=outside,
tick pos=left,
title={\(\displaystyle \mathbf{\alpha}\)=(33,69,71)/173},
x grid style={darkgray176},
xmin=-33.7114718614719, xmax=223.011471861472,
xtick style={color=black},
y grid style={darkgray176},
ymin=-1, ymax=190.3,
ytick style={color=black}
]
\path [draw=black, fill=black]
(axis cs:44,92)
--(axis cs:43.3214381724219,86.5715053793751)
--(axis cs:41.8284060356149,87.2855642274132)
--(axis cs:33.0676600666227,68.9676408377022)
--(axis cs:32.9323399333773,69.0323591622978)
--(axis cs:41.6930859023695,87.3502825520088)
--(axis cs:40.2000537655625,88.0643414000469)
--cycle;
\path [draw=black, fill=black]
(axis cs:22,46)
--(axis cs:22.6785618275781,51.4284946206249)
--(axis cs:24.1715939643851,50.7144357725868)
--(axis cs:32.9323399333773,69.0323591622978)
--(axis cs:33.0676600666227,68.9676408377022)
--(axis cs:24.3069140976305,50.6497174479912)
--(axis cs:25.7999462344375,49.9356585999531)
--cycle;
\path [draw=black, fill=black]
(axis cs:0,173)
--(axis cs:3.9975407307005,169.265208425308)
--(axis cs:2.54325344297477,168.475225207284)
--(axis cs:44.0659042577519,92.0357998437171)
--(axis cs:43.9340957422481,91.9642001562829)
--(axis cs:2.41144492747092,168.40362551985)
--(axis cs:0.95715763974519,167.613642301826)
--cycle;
\path [draw=black, fill=black]
(axis cs:88,11)
--(axis cs:84.0024592692995,14.734791574692)
--(axis cs:85.4567465570252,15.5247747927159)
--(axis cs:43.9340957422481,91.9642001562829)
--(axis cs:44.0659042577519,92.0357998437171)
--(axis cs:85.5885550725291,15.5963744801501)
--(axis cs:87.0428423602548,16.3863576981739)
--cycle;
\path [draw=black, fill=black]
(axis cs:173,0)
--(axis cs:167.630890295922,-1.0496003932785)
--(axis cs:167.843295533356,0.591712805081359)
--(axis cs:87.9903743850105,10.9256202478085)
--(axis cs:88.0096256149895,11.0743797521915)
--(axis cs:167.862546763335,0.740472309464428)
--(axis cs:168.07495200077,2.38178550782428)
--cycle;
\path [draw=black, fill=black]
(axis cs:3,22)
--(axis cs:8.36910970407847,23.0496003932785)
--(axis cs:8.15670446664367,21.4082871949186)
--(axis cs:88.0096256149895,11.0743797521915)
--(axis cs:87.9903743850105,10.9256202478085)
--(axis cs:8.13745323666468,21.2595276905356)
--(axis cs:7.92504799922988,19.6182144921757)
--cycle;
\path [draw=black, fill=black]
(axis cs:0,173)
--(axis cs:3.21867733001072,168.576300615404)
--(axis cs:1.64118724760107,168.075750877716)
--(axis cs:16.5714874659702,121.022683522856)
--(axis cs:16.4285125340298,120.977316477144)
--(axis cs:1.49821231566062,168.030383832004)
--(axis cs:-0.0792777667490323,167.529834094317)
--cycle;
\path [draw=black, fill=black]
(axis cs:33,69)
--(axis cs:29.7813226699893,73.423699384596)
--(axis cs:31.3588127523989,73.9242491222837)
--(axis cs:16.4285125340298,120.977316477144)
--(axis cs:16.5714874659702,121.022683522856)
--(axis cs:31.5017876843394,73.9696161679956)
--(axis cs:33.079277766749,74.4701659056832)
--cycle;
\path [draw=black, fill=black]
(axis cs:16.5,121)
--(axis cs:17.5128986581507,115.623845583661)
--(axis cs:15.8730747300374,115.847457937495)
--(axis cs:9.07431226260332,65.989866509645)
--(axis cs:8.92568773739668,66.010133490355)
--(axis cs:15.7244502048308,115.867724918205)
--(axis cs:14.0846262767175,116.091337272039)
--cycle;
\path [draw=black, fill=black]
(axis cs:1.5,11)
--(axis cs:0.48710134184926,16.3761544163385)
--(axis cs:2.12692526996257,16.1525420625049)
--(axis cs:8.92568773739668,66.010133490355)
--(axis cs:9.07431226260332,65.989866509645)
--(axis cs:2.27554979516922,16.1322750817949)
--(axis cs:3.91537372328253,15.9086627279613)
--cycle;
\path [draw=black, fill=black]
(axis cs:0,0)
--(axis cs:-1.01289865815074,5.37615441633854)
--(axis cs:0.626925269962574,5.15254206250491)
--(axis cs:1.42568773739668,11.010133490355)
--(axis cs:1.57431226260332,10.989866509645)
--(axis cs:0.775549795169219,5.13227508179491)
--(axis cs:2.41537372328253,4.90866272796128)
--cycle;
\path [draw=black, fill=black]
(axis cs:3,22)
--(axis cs:4.01289865815074,16.6238455836615)
--(axis cs:2.37307473003743,16.8474579374951)
--(axis cs:1.57431226260332,10.989866509645)
--(axis cs:1.42568773739668,11.010133490355)
--(axis cs:2.22445020483078,16.8677249182051)
--(axis cs:0.584626276717466,17.0913372720387)
--cycle;
\path [draw=black, fill=black]
(axis cs:0,0)
--(axis cs:-1.01289865815074,5.37615441633854)
--(axis cs:0.626925269962574,5.15254206250491)
--(axis cs:4.42568773739668,33.010133490355)
--(axis cs:4.57431226260332,32.989866509645)
--(axis cs:0.775549795169219,5.13227508179491)
--(axis cs:2.41537372328253,4.90866272796128)
--cycle;
\path [draw=black, fill=black]
(axis cs:9,66)
--(axis cs:10.0128986581507,60.6238455836615)
--(axis cs:8.37307473003743,60.8474579374951)
--(axis cs:4.57431226260332,32.989866509645)
--(axis cs:4.42568773739668,33.010133490355)
--(axis cs:8.22445020483078,60.8677249182051)
--(axis cs:6.58462627671747,61.0913372720387)
--cycle;
\path [draw=black, fill=black]
(axis cs:1.5,11)
--(axis cs:0.48710134184926,16.3761544163385)
--(axis cs:2.12692526996257,16.1525420625049)
--(axis cs:2.92568773739668,22.010133490355)
--(axis cs:3.07431226260332,21.989866509645)
--(axis cs:2.27554979516922,16.1322750817949)
--(axis cs:3.91537372328253,15.9086627279613)
--cycle;
\path [draw=black, fill=black]
(axis cs:4.5,33)
--(axis cs:5.51289865815074,27.6238455836615)
--(axis cs:3.87307473003743,27.8474579374951)
--(axis cs:3.07431226260332,21.989866509645)
--(axis cs:2.92568773739668,22.010133490355)
--(axis cs:3.72445020483078,27.8677249182051)
--(axis cs:2.08462627671747,28.0913372720387)
--cycle;
\path [draw=black, fill=black]
(axis cs:0,0)
--(axis cs:0.678561827578117,5.42849462062493)
--(axis cs:2.17159396438511,4.71443577258681)
--(axis cs:21.9323399333773,46.0323591622978)
--(axis cs:22.0676600666227,45.9676408377022)
--(axis cs:2.30691409763046,4.6497174479912)
--(axis cs:3.79994623443745,3.93565859995308)
--cycle;
\path [draw=black, fill=black]
(axis cs:44,92)
--(axis cs:43.3214381724219,86.5715053793751)
--(axis cs:41.8284060356149,87.2855642274132)
--(axis cs:22.0676600666227,45.9676408377022)
--(axis cs:21.9323399333773,46.0323591622978)
--(axis cs:41.6930859023695,87.3502825520088)
--(axis cs:40.2000537655625,88.0643414000469)
--cycle;
\addplot [semithick, red, mark=*, mark size=0.75, mark options={solid}, only marks]
table {%
173 0
};
\addplot [semithick, red, mark=*, mark size=0.75, mark options={solid}, only marks]
table {%
0 173
};
\addplot [semithick, red, mark=*, mark size=0.75, mark options={solid}, only marks]
table {%
0 0
};
\addplot [semithick, green, mark=*, mark size=0.75, mark options={solid}, only marks]
table {%
33 69
};
\addplot [semithick, blue, mark=*, mark size=0.75, mark options={solid}, only marks]
table {%
44 92
};
\addplot [semithick, blue, mark=*, mark size=0.75, mark options={solid}, only marks]
table {%
88 11
};
\addplot [semithick, blue, mark=*, mark size=0.75, mark options={solid}, only marks]
table {%
16.5 121
};
\addplot [semithick, blue, mark=*, mark size=0.75, mark options={solid}, only marks]
table {%
9 66
};
\addplot [semithick, blue, mark=*, mark size=0.75, mark options={solid}, only marks]
table {%
1.5 11
};
\addplot [semithick, blue, mark=*, mark size=0.75, mark options={solid}, only marks]
table {%
4.5 33
};
\addplot [semithick, blue, mark=*, mark size=0.75, mark options={solid}, only marks]
table {%
3 22
};
\addplot [semithick, blue, mark=*, mark size=0.75, mark options={solid}, only marks]
table {%
22 46
};
\end{axis}

\end{tikzpicture}}}
     \caption{Some examples of minimal cardinality $\ba$-mediated graphs for $d=3$.\label{plots:d3}}
 \end{figure}

 \begin{figure}[h]
 \fbox{\includegraphics[width=0.3\textwidth,trim={3cm 2cm 3cm 0.8cm},clip]{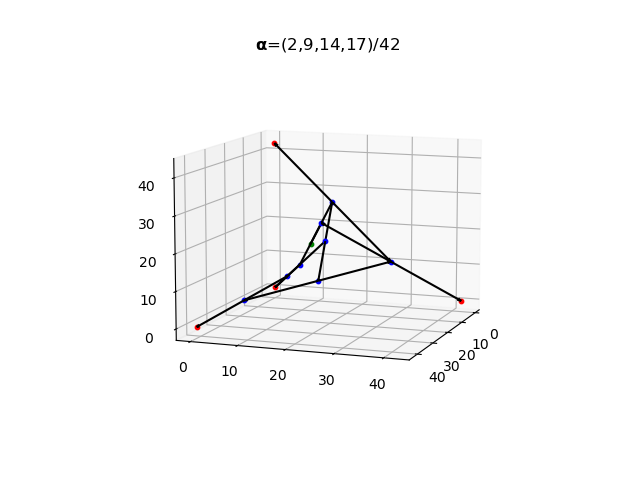}}~ \fbox{\includegraphics[width=0.3\textwidth,trim={3cm 2cm 3cm 0.8cm},clip]{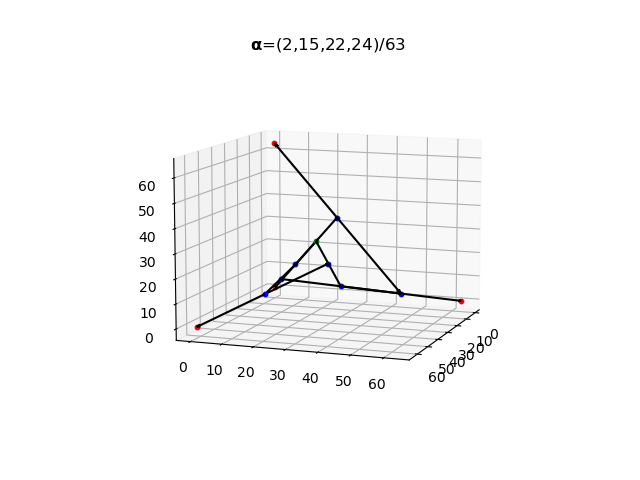}}~ \fbox{\includegraphics[width=0.3\textwidth,trim={3cm 2cm 3cm 0.8cm},clip]{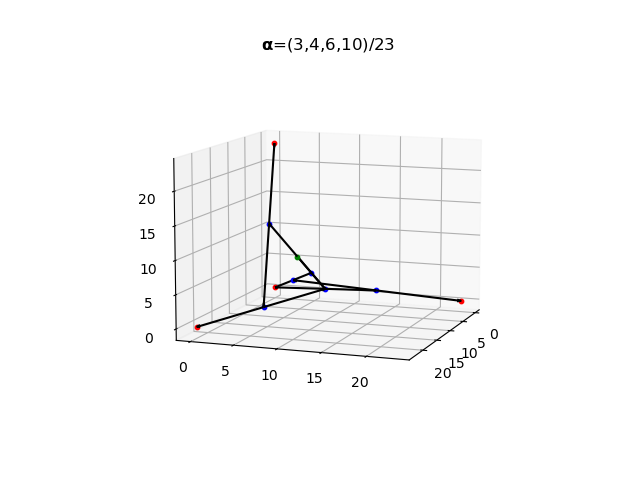}}\\
  \fbox{\includegraphics[width=0.3\textwidth,trim={3cm 2cm 3cm 0.8cm},clip]{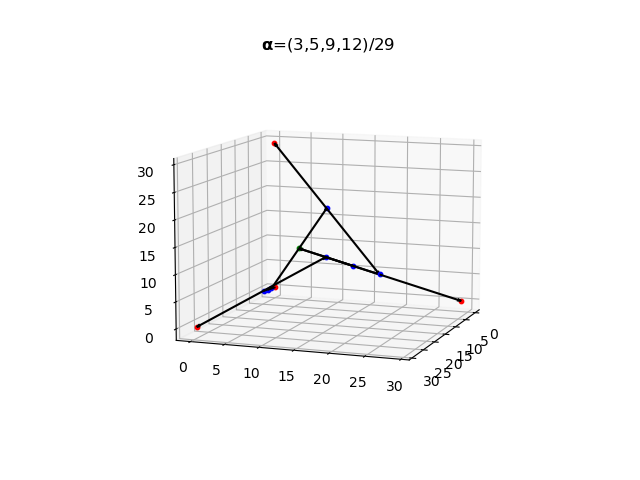}}~ \fbox{\includegraphics[width=0.3\textwidth,trim={3cm 2cm 3cm 0.8cm},clip]{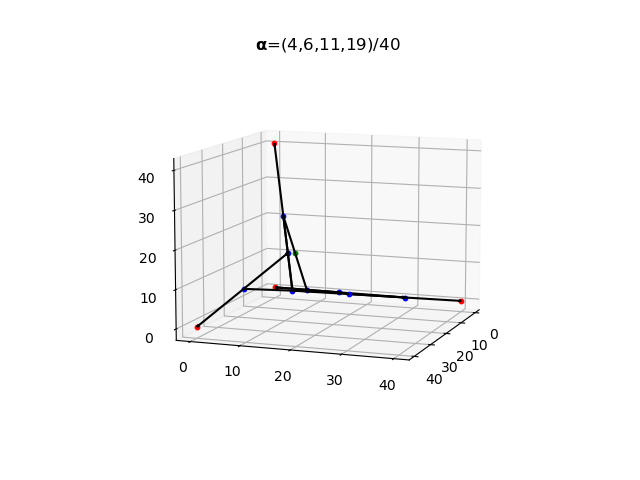}}~ \fbox{\includegraphics[width=0.3\textwidth,trim={3cm 2cm 3cm 0.8cm},clip]{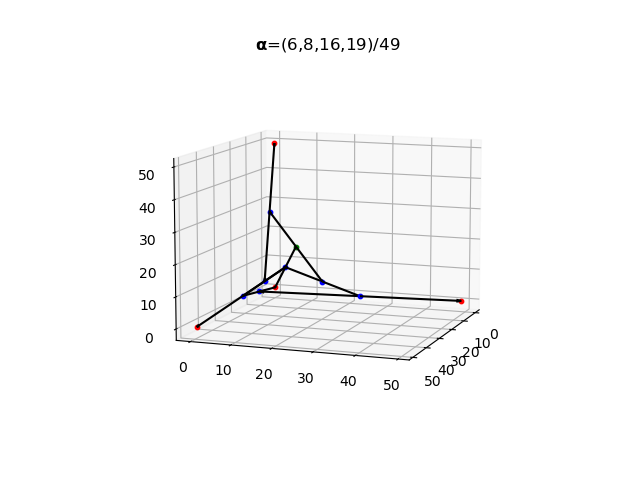}}
     \caption{Some examples of minimal cardinality $\ba$-mediated graphs for $d=4$.\label{plots:d4}}
 \end{figure}

\subsection{Computational experiments for the Gravitational Covering Location Problem}

The second set of experiments that we designed to test our approach is the use of a minimal extended representation of the $(p,\ba)$-power cone in a practical application. Specifically, we apply the representation to a continuous covering location problem whose coverage follows the Universal Law of Gravitation.

The goal of the problem that we propose is to optimally locate a given number of facilities in a $d$-dimensional space, $\R^d$, and decide their coverage radii based on different features to maximize the overall coverage of a given set of users.

This problem belongs to the family of Covering Location problems, that arise whenever a decision-maker aims to cover a given demand in case the facilities have a limited coverage area. In several practical situations, the services to be located are unable to satisfy the demand outside their coverage area but the demand covered by the services wants to be maximized. Church \cite{church1974maximal} introduces the Maximal Covering Location Problem (MCLP) as the problem of finding the positions of $p$ services on a given space, each of them endowed with a given coverage area, maximizing the covered demand of the users. The MCLP has attracted significant attention from both researchers and practitioners~\cite{murraywei}. Specifically, this type of problem has been applied in different fields, such as in the location of emergency services ~\cite{toregas1971location}, leak detection~\cite{blanco2024optimal}, mail advertising~\cite{dwyer1981branch}, archaeology~\cite{bell1985location}, among many others. The interested reader is referred to \cite{garcia2019covering} for further details and recent advances in covering location. 


To formulate the problem we first describe the parameters and decision variables that we use in our mathematical optimization model.

We are given a set of demand points $\{\a_i\}_{i\in I}\subset \R^d$ indexed by the set $I$.  We assume that $|J|$ facilities are to be located in $\R^d$, indexed by the set $J$. 

We assume that each demand point is allowed to be serviced by a facility if it belongs to a coverage area induced by a $\ell_p$-norm ball centered at the facility's coordinates. The radius of each facility has to be determined based on different features, such as the capacity of the facility, the satisfaction level, or the distance to certain points of interest. In particular, a demand point located at $\mathbf{a} \in \R^d$ can be covered by a facility located at $\x \in \R^d$ only if $\|\x - \mathbf{a}\|_p \leq G m_1^{\alpha_1} \cdots m_l^{\alpha_l}$, where $m_1, \ldots, m_l$ are the different features affecting the coverage radius, and $G \in \R_+, \ba =(\alpha_1, \ldots, \alpha_l) \in \Lambda_l$ are the aggregation parameters for these values. Note that the shape of this coverage area extends the notion of attraction between two bodies (in this case the demand points and the facilities) stated in the Universal Law of Gravitation by Newton~\cite{newton}, where the attraction of two objects can be obtained as:
$$
\frac{Gm_1m_2}{\|\x_1-\x_2\|^2_2}
$$
where $G$ is the gravity constant, $m_1$ and $m_2$ are the masses of the two bodies, and $\x_1$ and $\x_2$ their positions in the space. 

Observe that the assumption of the dependence of the coverage radii with different features has a practical interest in real-world situations where the decision-makers may decide to give a different service to the users based on their characteristics (certain profile of users of interest, preferred communities, serve minorities instead of just those closer to the facility, etc). Thus, the consideration of attraction-function measures, as the one generalizing Newton's law, would provide insights into the convenience of locating the service in different places, as well as which users are more favorable to the coverage. Specifically, one facility may sacrifice to cover certain users which are closer at the price of covering others that are further but provide better coverage.


The values of the features for the coverage radius could depend only on the demand point (i.e., certain characteristics of the specific demand point) or only on the facility (as the capacity level), or both (as the interest of a demand point for a given facility or viceversa). The values of these features are unknown, and determining their values informs the decision-maker about the importance that should be given to each of them in the coverage of demand points. A budget $B$ is available to invest in the features.

In addition, each demand point $i \in I$ is endowed with a weight $\omega_i \in \R_+$, which models the population of the region identified with the demand point or the profit for covering it.

The goal of the Gravitational Maximal Covering  Location problem is to determine the optimal positions of the facilities indexed in $J$ that maximize the overall profit for covering the demand $I$ such that the cost for the values of the different features does not exceed the given budget.

To derive our mathematical programming formulation for the problem, we use the following sets of decision variables:

$$
y_{ij} = \begin{cases}
1 & \mbox{ if demand point $i$ is accounted as covered by  facility $j$,}\\
0 & \mbox{otherwise}
\end{cases} \quad \forall i \in I, j \in J.
$$

$$
\x_j \in \R^d: \text{ Coordinates of the placement of the facilities, } \quad \forall j \in J.
$$
$$
m_{ijk} \in \R_+: \text{ Value of feature $k$ affecting the coverage of demand point $i$ by facility $j$.}
$$
for $i \in I, j \in J, k=1, \ldots, l$.


The Gravitational Covering Location Problem can be formulated as the following Mathematical Optimization Model:

\begin{align}
\max & \dsum_{i\in I}\sum_{j\in J} \omega_i y_{ij} \label{m0:obj1}\\
\mbox{s.t. } & \|\x_j - \a_i\|_p \leq G_{ij} m_{ij1}^{\alpha_1} \cdots m_{ijl}^{\alpha_l} + M (1-y_{ij}), \forall i\in I, j \in J,\label{m0:c1}\\
& \dsum_{j\in J} y_{ij}\leq 1, \forall i \in I,\label{m0:c2}\\
& \dsum_{i\in I}\dsum_{j\in P} \dsum_{k=1}^l m_{ijk}\leq B, \label{m0:3}\\
& \x_j \in \R^d, \forall j\in J,\nonumber\\
& y_{ij} \in \{0,1\}, \forall i \in I, j\in J,\nonumber\\
& m_{ijk} \in \R_+, \forall i\in I, j\in J, k=1, \ldots, l.\nonumber
\end{align}

The objective function model is the overall weighted coverage of the solution. Constraints \eqref{m0:c1} enforce that in case the demand point $i$ is covered by facility $j$, $\a_i$ must belong to the coverage area determined by the different features. Otherwise, the value of the parameter $M> \max_{i, i' \in I} \|a_i-a_{i'}\|_p$ makes this constraint redundant. Constraints \eqref{m0:c2} state that each demand point can be accounted as covered by at most one facility. Constraint \eqref{m0:3} is the budget constraint. Finally, the domains of the variables are detailed.

Observe that constraints \eqref{m0:c1} can be written as conic constraints in the shape of those in Remark \ref{rem:12}, i.e., Constraints \eqref{m0:c1} are equivalent to:
$$
(\x,\y, \textbf{m})\in \Q_p(\ba)[f_{ij},g_{ij},h_{ij}], \forall i\in I, j\in J,
$$
where $f_{ij}(\x)=\x_j-\a_i$,  $h_{ij}(\textbf{m})= G_{ij} \cdot  \left(m_{ij1}, \ldots, m_{ijl}\right)$, and $g_{ij}(\y)=M(1-y_{ij})$

Note that one can incorporate more sophisticated expressions for the $m$-variables in the form of linear or conic constraints keeping a similar structured problem. 

We have run a series of experiments to analyze the performance of solving this problem with an optimal SOC representation of the generalized power cone with respect to the usual representation that provides the upper bound described in Remark \ref{bounds}. The chosen inputs are a random generation of four planar datasets of demand points in the plane with $|I| \in \{25, 50, 75, 100\}$. The parameters $\omega_i$ were randomly chosen as integers in $[0,10]$. We assume that the $m$-variables are bounded in $[0,1]$ where $l=3$; that the budget, $B$, is $\gamma(2|I|+|J|)$, where $\gamma=\frac{1}{4}$; and the constants $G_{ij}$ are all equal to $1$. We solved the problem for $|J| \in \{2, 3, 5\}$. 

We run the experiments for all the combinations of $\ell_p$-norms and $\ba$ values described in Table \ref{t:values}:

\setlength\extrarowheight{7pt}

\begin{table}[H]
\begin{center}
\begin{tabular}{c|c}
$p$ & $\ba$\\\hline
$\frac{43}{31}$ & $\left(\frac{2}{26},\frac{5}{26},\frac{19}{26}\right)$\\
\multirow{2}{*}{$2$} & $\left(\frac{13}{80},\frac{33}{80},\frac{34}{80}\right)$\\
& $\left(\frac{6}{60},\frac{19}{60},\frac{35}{60}\right)$\\
$\frac{17}{3}$ &$\left(\frac{35}{180},\frac{58}{180},\frac{87}{180}\right)$\\
\multicolumn{2}{c}{}\\
\end{tabular}
\caption{Norms and $\ba$ used in our computational experiments.\label{t:values}}
\end{center}
\end{table}

In total, 144 instances were generated. A time limit of 2 hours was fixed for all the instances.

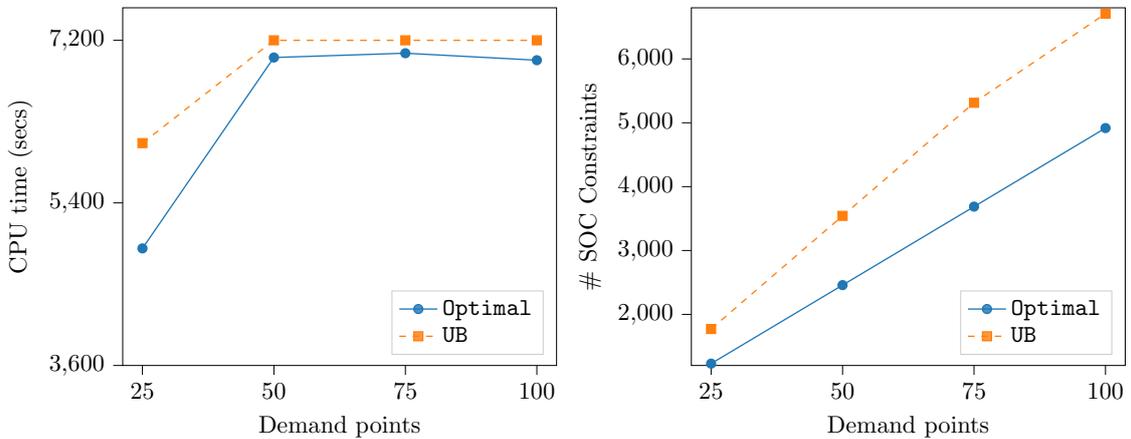
\begin{figure}[H]
\centering
\adjustbox{width=\textwidth}{
\begin{tikzpicture}

\definecolor{darkgray176}{RGB}{176,176,176}
\definecolor{darkorange25512714}{RGB}{255,127,14}
\definecolor{lightgray204}{RGB}{204,204,204}
\definecolor{steelblue31119180}{RGB}{31,119,180}

\begin{axis}[
legend cell align={left},
legend style={
  fill opacity=0.8,
  draw opacity=1,
  text opacity=1,
  at={(0.97,0.03)},
  anchor=south east,
  draw=lightgray204,
},
xtick={25,50,75,100},
ytick={1800, 3600, 5400, 7200},
tick align=outside,
tick pos=left,
x grid style={darkgray176},
xlabel={Demand points},
ylabel={CPU time (secs)},
xmin=21.25, xmax=103.75,
xtick style={color=black},
y grid style={darkgray176},
ymin=3600,
ytick style={color=black}
]

\addplot [semithick, steelblue31119180, mark=*, mark size=2, mark options={solid}]
table {%
25 4896.06
50 7009.53
75 7057.13
100 6979.62
};
\addlegendentry{\texttt{Optimal}}
\addplot [semithick, darkorange25512714, dashed, mark=square*, mark size=2, mark options={solid}]
table {%
25 6061.68
50 7200
75 7200
100 7200
};
\addlegendentry{\texttt{UB}}
\end{axis}

\end{tikzpicture}~
\begin{tikzpicture}

\definecolor{darkgray176}{RGB}{176,176,176}
\definecolor{darkorange25512714}{RGB}{255,127,14}
\definecolor{lightgray204}{RGB}{204,204,204}
\definecolor{steelblue31119180}{RGB}{31,119,180}

\begin{axis}[
legend cell align={left},
legend style={
  fill opacity=0.8,
  draw opacity=1,
  text opacity=1,
    at={(0.97,0.03)},
  anchor=south east,
  draw=lightgray204
},
xtick={25,50,75,100},
ytick={1000, 2000, 3000, 4000, 5000, 6000, 7000},
ylabel={\# SOC Constraints},
tick align=outside,
tick pos=left,
x grid style={darkgray176},
xlabel={Demand points},
xmin=21.25, xmax=103.75,
xtick style={color=black},
y grid style={darkgray176},
ymin=1200, ymax=6800,
ytick style={color=black}
]
\addplot [semithick, steelblue31119180, mark=*, mark size=2, mark options={solid}]
table {%
25 1229
50 2458
75 3687
100 4917
};
\addlegendentry{\texttt{Optimal}}
\addplot [semithick, darkorange25512714, dashed, mark=square*, mark size=2, mark options={solid}]
table {%
25 1771
50 3542
75 5313
100 6708
};
\addlegendentry{\texttt{UB}}
\end{axis}

\end{tikzpicture}}
\caption{Average CPU times (left) and number of constraints (right) of the optimal SOC representations of the Gravitational Maximal Covering Location Problem with respect to \texttt{UB} SOC representantions.\label{fig:lines}}
\end{figure}

In figures \ref{fig:devs1} and \ref{fig:lines} we summarize the main results obtained in our experiments. 
 In Figure \ref{fig:lines} we represent, for each number of demand points in the experiments, in the left plot, the average CPU time (in seconds) required by Gurobi to solve the problem (or 7200 seconds if the time limit was reached without guarantying optimality), and, in the right plot, the number of SOC constraints in the formulation, both for the optimal representation (\texttt{Optimal}) and the one that provides the upper bound (\texttt{UB}). As can be observed, using the optimal cones is advisable since it reduces, significantly, the computational resources required to solve the problems. Specifically, the upper bound representation was able to solve optimally only 9 out of the 144 instances within the time limit (all of them form 25 demand points), whereas the optimal representation was able to solve 21. The main reason for such a gain is the number of SOC constraints in the models (on average 30\% less of these constraints in the optimal representation model). 

 Since most of the instances were not optimally solved within the time limit, in what follows we compare the quality of the best solution obtained after two hours. In the obtained results, the optimal model always reaches a coverage at least as good as the one obtained with the upper bound representation. In 94 out of the 144 instances ($65\%$ of the instances), the optimal representation model obtained a strictly better coverage than the upper bound model. 
 
 In Figure \ref{fig:devs1} we represent the percent deviation, in coverage, of the best solution obtained with the optimal representation with respect to the one obtained with the upper bound representation for those instances where the solutions are different. They represent $55\%$ of the instances for $|I|=25$, $72\%$ of the instances for $|I|=50$, $64\%$ of the instances for $|I|=75$, and $69\%$ of the instances for $|I|=100$. Observe that the average gain is impressive in the quality of the solution. Although the average deviation is $26\%$, in some instances the optimal representation obtained solutions with $82\%$ more coverage than those obtained with the upper bound representation.

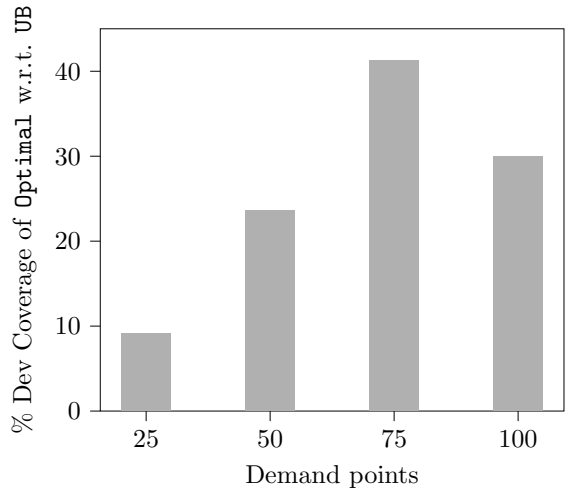
\begin{figure}[H]
\centering
\adjustbox{width=0.5\textwidth}{
\begin{tikzpicture}

\definecolor{darkgray176}{RGB}{176,176,176}
\definecolor{green01270}{RGB}{0,127,0}

\begin{axis}[
tick align=outside,
tick pos=left,
x grid style={darkgray176},
xlabel={Demand points},
xmin=-0.37, xmax=3.37,
xtick style={color=black},
xtick={0,1,2,3},
xticklabels={25,50,75,100},
y grid style={darkgray176},
ylabel={\% Dev Coverage of \texttt{Optimal}  w.r.t. \texttt{UB}},
ymin=0, ymax=45,
ytick style={color=black}
]
\draw[draw=none,fill=darkgray176] (axis cs:-0.2,0) rectangle (axis cs:0.2,9.13);
\draw[draw=none,fill=darkgray176] (axis cs:0.8,0) rectangle (axis cs:1.2,23.58);
\draw[draw=none,fill=darkgray176] (axis cs:1.8,0) rectangle (axis cs:2.2,41.22);
\draw[draw=none,fill=darkgray176] (axis cs:2.8,0) rectangle (axis cs:3.2,30.00);
\end{axis}

\end{tikzpicture}}
\caption{Average coverage deviations of the optimal SOC representations of Gravitational Maximal Covering Location Problem with respect to \texttt{UB} SOC representantions.\label{fig:devs1}}
\end{figure}

\section*{Conclusions}\label{sec:6}

In this paper, we analyze extended representations of general cones, the so-called $(p,\ba)$-power cones. These cones are highly practical when modeling constraints in optimization problems. We study different representations for general values for $\ba$ and $p$ as simpler cones and study further simplifications to second order cones in case $\ba$ and $p$ have rational entries. Specifically, we provide a novel mathematical optimization-based procedure to construct minimal representations, i.e. using the minimum number of extra variables and constraints in the representation. The idea behind our approach is the identification of each extended representation with a graph, the mediated graph, whose size coincides with the complexity of the representation. Thus, we introduce, for the first time, the Minimum Cardinality Mediated Graph Problem (MCMGP) and derive a suitable mixed integer linear programming formulation for it. We tested the approach both based on the computational difficulty of solving the $\ba$-MCMGP and the ability to use the minimal representations in a facility location problem compared to the representation provided by the upper bound (Remark \ref{bounds}).  The main conclusion of our computational experience is that the use of minimal representations of generalized is advisable in practice and that its study is worth it, not only for the mathematical challenges that it poses but also for its practical application. 

Further research on the topic includes a deeper study of the minimal mediated graphs that induce a minimal representation of a cone. The geometric properties of these intriguing graphs would result in a direct impact, in the form of valid inequalities for the $\ba$-MCMGP, in the solution times required to solve the problem. The analysis of particular values of $p$ and $\ba$ whose representation can be \textit{easily} derived would be also useful. We conjecture that, in general, the $\ba$-MCMGP is NP-hard. This fact would imply that finding a minimal representation of these cones is a challenging combinatorial problem. The proof of this conjecture will be the topic of our forthcoming research.

\section*{Acknowledgements}

This research has been partially supported by Spanish Ministerio de Ciencia e Innovación, AEI/FEDER grant number PID 2020 - 114594GBC21, RED2022-134149-T (Thematic Network on Location Science and Related Problems), Junta de Andalucía projects B-FQM-322-UGR20 and AT 21\_00032, and the IMAG-Maria de Maeztu grant CEX2020-001105-M /AEI /10.13039/501100011033. The authors would like to thank  Prof. Justo Puerto for interesting discussions on previous versions of this work and  the “Centro de Supercomputaci\'on” of the
Universidad de Granada for providing us support to use the \texttt{albaic\'in} server.

\end{document}